\documentclass[12pt,twoside,reqno]{amsart}
\usepackage[T1]{fontenc}
\usepackage[utf8]{inputenc} 
\usepackage{tikz}
\usetikzlibrary{calc}
\usetikzlibrary{shapes.geometric, arrows}
\usepackage[colorlinks=true,citecolor=blue]{hyperref}
\usepackage{mathptmx, amsmath, amssymb, amsfonts, amsthm, mathptmx, enumerate, color}
\usepackage{graphicx}
\usepackage{epstopdf}

\newtheorem{theorem}{Theorem}[section]

\theoremstyle{definition}

\newtheorem{remark}{Remark}[section]

\numberwithin{equation}{section}

\newcommand{\poubelle}[1]{}
\begin{document}
\setcounter{page}{1}

\vspace*{1.0cm}
\title[Optimal Control of Epidemics Models]
{Optimal Control of an Impulsive VS-EIAR Epidemic Model with Application to COVID-19}
\author[M. A. Diop, M. Elghandouri \& K. Ezzinbi]{Mamadou Abdoul DIOP$^{1,4}$, Mohammed Elghandouri$^{2,3,4,*}$, Khalil Ezzinbi$^{2,4}$}
\maketitle
\vspace*{-0.6cm}

\begin{center}
{\footnotesize {\it

$^1$Department of Mathematics, Faculty of Applied Sciences and Technology, Gaston Berger University, Senegal.\\
$^2$Departement of Mathematics, Faculty of Sciences Semlalia, Cadi Ayyad University, Marrakesh, Morocco.\\
$^3$Centre INRIA de Lyon, CEI-2 56, Boulevard Niels Bohr, 69 603, Villeurbanne, France.\\
$^4$IRD-UMMISCO, 32 Av Henri Varagnat, 93 143 Bondy, France.
}}\end{center}

\vskip 4mm {\small\noindent {\bf Abstract.}
	In this work, we investigate a VS-EIAR epidemiological model that incorporates vaccinated individuals $\{V_i : i = 1, \ldots, n\}$, where $n \in \mathbb{N}^{*}$. The dynamics of the VS-EIAR model are governed by a system of ordinary differential equations describing the evolution of vaccinated, susceptible, exposed, infected, asymptomatic, and deceased population groups. Our primary objective is to minimize the number of susceptible, exposed, infected, and asymptomatic individuals by administering vaccination doses to susceptible individuals and providing treatment to the infected population. To achieve this, we employ optimal control theory to regulate the epidemic dynamics within an optimal terminal time $\tau^{*}$. Using Pontryagin's Maximum Principle (PMP), we establish the existence of an optimal control pair $(v^{*}(t), u^{*}(t))$. Additionally, we extend the model to an impulsive VS-EIAR framework, with particular emphasis on the impact of immigration and population movement. Finally, we present numerical simulations to validate the theoretical results and demonstrate their practical applicability.}


\renewcommand{\thefootnote}{}
\footnotetext{ $^*$Corresponding author.
\par
E-mail addresses: mamadou-abdoul.diop@ugb.edu.sn (M. A. Diop), medelghandouri@gmail.com (M. Elghandouri), ezzinbi@uca.ac.ma (K. Ezzinbi).
\par
Received xx xx, 20xx; Accepted xx xx, 20xx.}

\section{Introduction}

COVID-19 has emerged as one of the most formidable global challenges, profoundly impacting economies, societies, and political systems worldwide. The World Health Organization (WHO) officially identified the first case in Wuhan, China, on December 31, 2019 \cite{3}. Common respiratory symptoms of COVID-19 include cough, fever, breathing difficulties, and shortness of breath. In severe cases, the infection can progress to life-threatening conditions such as pneumonia, severe acute respiratory syndrome, respiratory failure, and even death, as reported by the World Health Organization (WHO). In the early stages of the pandemic, quarantine and treatment were the primary measures used to curb the spread of COVID-19. However, these measures incurred significant economic costs and exacerbated existing crises, resulting in prolonged recovery periods. Fortunately, the development of vaccines has provided a more sustainable solution, enabling effective control of COVID-19 transmission and reducing reliance on stringent quarantine measures. In this study, we explore vaccination as a control strategy from a mathematical perspective. We propose a comprehensive mathematical model that advances existing literature by capturing the dynamics of COVID-19 more accurately. Through this refined model, we aim to enhance the understanding of the pandemic's spread and provide insights into effective control strategies. Abbasi et al. \cite{11} proposed an impulsive SQEIAR epidemic model to control COVID-19 transmission using two control strategies: quarantine for susceptible populations and treatment for infected individuals. Araz \cite{8} developed a comprehensive mathematical model that examines COVID-19 transmission scenarios, incorporating stability analysis, optimal control strategies, and the positiveness of solutions. In \cite{12}, the authors introduced a nonlinear deterministic model to study COVID-19 controllability using Pontryagin's Maximum Principle (PMP). They identified four time-dependent optimal control actions: $u_1^{*}$ (social distancing), $u_2^{*}$ (surface cleaning), $u_3^{*}$ (precautionary measures for exposed individuals), and $u_4^{*}$ (fumigation of public spaces). G. B. Libotte et al. \cite{7} presented an SIR model to determine the most effective vaccination strategy, proposing two optimal control approaches: one to reduce infected individuals during treatment and another to minimize both infections and vaccine concentration. They also addressed an inverse problem using Differential Evolution and Multi-objective Optimization Differential Evolution algorithms to estimate SIR model parameters. Shen and Chou \cite{10} introduced a novel optimal control model with four strategies: prevention measures, vaccine control, rapid screening of exposed individuals, and management of non-screened infected cases. Further strategies and insights are explored in \cite{16, 18, 14, 17, 13, 15, 6, 9, 19} and the references therein.

In \cite{11}, the authors aim to minimize the number of susceptible, exposed, infected, and asymptomatic individuals while maximizing quarantined and recovered populations using optimal control theory. This is achieved by minimizing a cost functional $\mathcal{J}$ associated with treatment $U(t) \in [0,1]$ and quarantine $\lambda(t) \in [0,1]$ over an optimal time interval. Pontryagin's Maximum Principle (PMP) is employed to prove the existence of an optimal control $(U^{*}(t), \lambda^{*}(t))$ that minimizes $\mathcal{J}$ \cite[Sections 2 and 4]{11}. However, given the high costs of quarantine and the availability of multiple COVID-19 vaccines, our study proposes an alternative strategy. We introduce an optimal control VS-EIAR epidemic model that replaces quarantine with vaccination. The approach involves administering $n$ vaccine doses to susceptible individuals, assuming that recovered individuals gain temporary immunity or immunity to the same virus variant. Furthermore, we assume that individuals receiving the maximum number of doses exhibit negligible or null infection rates ($\delta_{n} \simeq 0$). Additionally, the number of individuals vaccinated with the $i^{th}$ dose is assumed to be greater than those vaccinated with the $(i+1)^{th}$ dose. Further details are provided in Section \ref{section 2}, with a visual representation in Fig \ref{Fig 1}.

In the impulsive case, the model incorporates the effects of immigration and travel on population dynamics. Impulsive epidemiological models are biologically significant, as they enhance the accessibility and applicability of epidemic modeling by introducing sudden changes or interventions. While the literature on this topic is extensive, it is beyond the scope of this discussion to cover all relevant aspects. Interested readers may refer to works such as \cite{21,24,23,27,26,20,22,25} and the references therein for further details. Notably, Agarwal et al. \cite{27} examined the controllability of a generalized time-varying delay SEIR epidemic model using impulsive vaccination controls, demonstrating that impulsive vaccination can reduce discrepancies between the SEIR model and its reference model. Hui and Chen \cite{26} investigated impulsive vaccination strategies in SIR models, proving their superior applicability and effectiveness compared to classical vaccination approaches. Wang et al. \cite{25} analyzed an impulsive epidemiological model for pest control, showing that susceptible pest eradication is globally stable when the impulsive interval is below a critical threshold. In this study, we incorporate an impulsive population component to model the spread of COVID-19 through immigration and travel. Further details are provided in Section \ref{section 4}, with a visual representation in Fig \ref{Fig 2}.

In summary, this work is organized as follows. Section \ref{section 2} provides a detailed description of the VS-EIAR epidemic model and presents the mathematical framework governing its dynamics. Section \ref{section 3} discusses the optimal control strategy applied to the proposed VS-EIAR model. In Section \ref{section 4}, we explore the impulsive VS-EIAR dynamic model, focusing on the impact of immigration and travel. Section \ref{section 5} establishes the existence of an optimal control within an optimal time interval using Pontryagin's Maximum Principle. An application to COVID-19 is presented in Section \ref{section 6}, followed by a comparative analysis of three diseases (COVID-19, Ebola, and Influenza) in Section \ref{section 8}. Finally, Section \ref{section 7} concludes the study, and an appendix is included for additional details.

\section{VS-EIAR Epidemic Model} \label{section 2}
\noindent

In this section, we introduce a VS-EIAR epidemic model aimed at controlling disease propagation within a short time-frame. Building on the SEIAR epidemic model from \cite{11}, which excludes natural mortality and births, we propose a nonlinear VS-EIAR model comprising $n+6$ non-negative state variables: $V_1(t), \ldots, V_n(t), S(t), E(t), A(t), I(t), R(t),$ and $D(t)$. Here, $V_i(t)$ (for $i = 1, \ldots, n$) represents the number of individuals who have received $i$ vaccine dose(s) at time $t$ but have not yet received the $(i+1)^{th}$ dose. $S(t)$ denotes the susceptible population at risk of infection. Upon infection, susceptible individuals transition to the exposed group, $E(t)$, which includes individuals infected but not yet infectious. Exposed individuals may become infectious at rate $k$, joining either the asymptomatic group, $A(t)$, or the symptomatic group, $I(t)$. A fraction $z$ of exposed individuals move to the symptomatic group, while the remainder transition to the asymptomatic group. From the asymptomatic group, a fraction $p$ recovers, joining $R(t)$, while the remaining $(1-p)$ become symptomatic. A fraction $(1-\alpha)$ of the symptomatic population dies due to infection, with the remainder recovering. Following recommendations from the World Health Organization (WHO) and viral disease specialists, vaccination is targeted exclusively at susceptible individuals, excluding those currently infected or recently recovered. Figure \ref{Fig 1} illustrates the biological dynamics of the proposed model.
\begin{figure}[ht]
	\centering
	\scalebox{0.9}{
		\begin{tikzpicture}[node distance=2.9cm, auto, thick, >=stealth]
			
			\node [ellipse, draw, fill=green!60, align=center, minimum width=1.4cm, minimum height=1.4cm] (R) {R};
			
			\node [ellipse, draw, fill=red!60, below of=R, align=center, minimum width=1.4cm, minimum height=1.4cm] (I) {I};
			
			\node [ellipse, draw, fill=red!30, left of=I, align=center, minimum width=1.4cm, minimum height=1.4cm] (A) {A};
			
			\node [ellipse, draw, fill=black!10, right of=I, align=center, minimum width=1.4cm, minimum height=1.4cm] (D) {D};
			
			\node [ellipse, draw, fill=yellow!60, below of=I, align=center, minimum width=1.4cm, minimum height=1.4cm] (E) {E};
			
			\node [ellipse, draw, fill=green!50, below of=E, align=center, minimum width=1.4cm, minimum height=1.4cm] (V2) {\(V_2\)};
			
			\node [ellipse, draw, fill=green!30, left of=V2, align=center, minimum width=1.4cm, minimum height=1.4cm] (V1) {\(V_1\)};
			
			\node [ellipse, draw, fill=blue!30, left of=V1, align=center, minimum width=1.4cm, minimum height=1.4cm] (S) {S};
			
			\node [ellipse, draw, fill=green!65, right of=V2, align=center, minimum width=1.4cm, minimum height=1.4cm] (Vn1) {\(V_{n-1}\)};
			
			\node [ellipse, draw, fill=green!80, right of=Vn1, align=center, minimum width=1.4cm, minimum height=1.4cm] (Vn) {\(V_n\)};
			
			\draw[->, line width=0.7pt, >=stealth, shorten <=0.5mm, shorten >=0.5mm] (S) -- (E) node[midway, sloped, above] {\(\scriptstyle \beta\Lambda\)};
			
			\draw[->, line width=0.7pt, >=stealth, shorten <=0.5mm, shorten >=0.5mm] (V1) -- (E) node[midway, sloped, above] {\(\scriptstyle \delta_1\)};
			
			\draw[->, line width=0.7pt, >=stealth, shorten <=0.5mm, shorten >=0.5mm] (V2) -- (E) node[midway, left] {\(\scriptstyle \delta_2\)};
			
			\draw[->, line width=0.7pt, >=stealth, shorten <=0.5mm, shorten >=0.5mm] (Vn1) -- (E) node[midway, sloped, above] {\(\scriptstyle \delta_{n-1}\)};
			
			\draw[->, line width=0.7pt, >=stealth, shorten <=0.5mm, shorten >=0.5mm] (Vn) -- (E) node[midway, sloped, above] {\(\scriptstyle \delta_{n}\)};
			
			\draw[->, line width=0.8pt,draw=green!50!black, >=stealth, shorten <=0.5mm, shorten >=0.5mm] (S) -- (V1) node[midway, sloped, below] {\(\scriptstyle \gamma_1v\)};
			
			\draw[->, line width=0.8pt,draw=green!50!black, >=stealth, shorten <=0.5mm, shorten >=0.5mm] (V1) -- (V2) node[midway, sloped, below] {\(\scriptstyle \gamma_2v\)};
			
			\draw[dashed,-, line width=0.7pt,draw=green!50!black, >=stealth, shorten <=0.5mm, shorten >=0.5mm] (V2) -- (Vn1) node[midway, sloped, below] {};
			
			\draw[->, line width=0.8pt,draw=green!50!black, >=stealth, shorten <=0.5mm, shorten >=0.5mm] (Vn1) -- (Vn) node[midway, sloped, below] {\(\scriptstyle \gamma_nv\)};
			
			\draw[->, line width=0.7pt, >=stealth, shorten <=0.5mm, shorten >=0.5mm] (E) -- (A) node[midway, sloped, above] {\(\scriptstyle (1-z)k\)};
			
			\draw[->, line width=0.7pt, >=stealth, shorten <=0.5mm, shorten >=0.5mm] (E) -- (I) node[midway, right] {\(\scriptstyle zk\)};
			
			\draw[->, line width=0.7pt, >=stealth, shorten <=0.5mm, shorten >=0.5mm] (A) -- (I) node[midway, sloped, above] {\(\scriptstyle (1-p)\eta\)};
			
			\draw[->, line width=0.7pt, >=stealth, shorten <=0.5mm, shorten >=0.5mm] (I) -- (D) node[midway, sloped, above] {\(\scriptstyle (1-\alpha)f\)};
			
			\draw[->, line width=0.7pt, >=stealth, shorten <=0.5mm, shorten >=0.5mm] (A) -- (R) node[midway, sloped, above] {\(\scriptstyle p\eta\)};
			
			\draw[->, line width=0.7pt, >=stealth, shorten <=0.5mm, shorten >=0.5mm] (I) -- (R) node[midway, left] {\(\scriptstyle \alpha f\)};
			
			\draw[->, draw=green!50!black, line width=0.7pt, shorten <=0.5mm, shorten >=0.5mm, bend right=45] (I) to node[midway, right] {\(\scriptstyle u\)} (R);
		\end{tikzpicture}
	}
	\caption{The VS-EIAR epidemic model excluding impulsive growth dynamics.}
	\label{Fig 1}
\end{figure}
 
 It is natural to assume that a susceptible individual cannot receive the $(i+1)^{th}$ dose without first receiving the $i^{th}$ dose. This assumption justifies the inequality:
 \begin{center}
 	$\gamma_1 \geq \gamma_2 \geq \cdots \geq \gamma_n \geq 0$ \quad ($\gamma_1 > 0$).
 \end{center}
 Additionally, we assume that individuals who have received more vaccine doses are less susceptible to infection than those who have received fewer doses. This is reflected in the inequality:
 \begin{center}
 	$\delta_1 \geq \delta_2 \geq \cdots \geq \delta_{n-1} \geq \delta_n \simeq 0$.
 \end{center}
 We further assume that $\gamma_i \geq \delta_i$ for $i = 1, 2, \ldots, n$, where $\beta, \eta, p, k, z, \alpha, f \in [0, 1]$ are constants.
 
 The primary objective is to employ optimal control theory to mitigate the spread of the epidemic by administering vaccination $v(t)$ to susceptible individuals and treatment $u(t)$ to infected individuals. The dynamics of the controlled model are governed by the following system of ordinary differential equations (ODEs):
\begin{equation}
\left\{\begin{array}{l}
\dot{V}_1(t)= \gamma_1v(t)S(t)-\gamma_2v(t)V_1(t)-\delta_1V_1(t)\\[12pt]
\left\{\begin{array}{l}
\dot{V}_{i}(t)=\gamma_{i}v(t)V_{i-1}(t)-\gamma_{i+1}v(t)V_{i}(t)-\delta_{i}V_{i}(t),\\[12pt]
\text{for}\hspace{0.1cm} i=2,3,\ldots,n-1,
\end{array}\right.\\[22pt]
\dot{V}_n(t)=\gamma_nv(t)V_{n-1}(t)\\[12pt]
\dot{S}(t)=-(\beta\Lambda(t)+\gamma_1v(t))S(t)\\[12pt]
\dot{E}(t)=\beta\Lambda(t)S(t)-kE(t)+\sum_{i=1}^{n-1}\delta_iV_i(t)\\[12pt]
\dot{A}(t)=(1-z)kE(t)-\eta A(t)\\[12pt]
\dot{I}(t)=zkE(t)+(1-p)\eta A(t)-fI(t)-u(t)I(t)\\[12pt]
\dot{R}(t)=\alpha fI(t)+u(t)I(t)+p\eta A(t)\\[12pt]
\dot{D}(t)= (1-\alpha)f I(t),
\end{array}\right.
\label{eq 1}
\end{equation}
for $t \in [0, \tau]$, where $\tau \in \mathbb{R}^{+}$ and $\Lambda(t) = \varepsilon E(t) + (1 - q) I(t) + \mu A(t)$. Here, $\varepsilon > 0$, $1 - q > 0$, and $\mu > 0$ represent the reduced transmissibility factors for exposed, infected, and asymptomatic individuals, respectively. The initial conditions are given by:
\[
(S(0), E(0), A(0), I(0), R(0), D(0), V_1(0), \ldots, V_n(0)) = (S_0, E_0, A_0, I_0, R_0, D_0, V_{1,0}, \ldots, V_{n,0}).
\]
Let $N(t) = S(t) + E(t) + A(t) + I(t) + R(t) + \sum\limits_{i=1}^{n} V_i(t)$ denote the total population at time $t$. 

The admissible control sets $U^{1}_{ad}$ and $U^{2}_{ad}$ are defined as:
\[
U^{1}_{ad} = \left\{ v \mid v \text{ is Lebesgue measurable and } v(t) \in \left[0, \frac{1}{\gamma_1}\right] \text{ for } t \in \mathbb{R}^{+} \right\},
\]
and
\[
U^{2}_{ad} = \left\{ u \mid u \text{ is Lebesgue measurable and } u(t) \in [0, 1] \text{ for } t \in \mathbb{R}^{+} \right\}.
\]
\subsection{Existence of Solutions}
\noindent

The following theorem guarantees the existence and uniqueness of solutions for equation \eqref{eq 1}. The proof is provided in the Appendix.

\begin{theorem} \label{theorem 2.1} 
	Let $V_{1,0} \geq 0$, $\ldots$, $V_{n,0} \geq 0$, $S_0 \geq 0$, $E_0 \geq 0$, $A_0 \geq 0$, $I_0 \geq 0$, $R_0 \geq 0$, and $D_0 \geq 0$. For fixed controls $v \in U^{1}_{ad}$ and $u \in U^{2}_{ad}$, equation \eqref{eq 1} admits a unique positive bounded solution defined on $\mathbb{R}^{+}$.
\end{theorem}

The basic reproduction number for the uncontrolled model (the SEIAR epidemic model) is given by:
\begin{equation}
	\mathcal{R}_0 = \beta N_0 \left[\frac{z}{\alpha f} + \frac{\mu(1 - z)}{\eta}\right].
\end{equation}
If $\mathcal{R}_0 < 1$, the infection dies out. However, if $\mathcal{R}_0 > 1$, an epidemic occurs, necessitating the implementation of control measures. For the COVID-19 example discussed in Section \ref{section 6}, we find $\mathcal{R}_0 = 1.52 > 1$ with $\beta = 5 \times 10^{-4}$, indicating the presence of the epidemic and the need for controls.

In general, the basic reproduction number increases as the transmission coefficient $\beta$ increases or the recovery rate from the infectious class decreases, signaling the potential for an epidemic. Furthermore, since $N'(t) = (\alpha - 1) I(t) \leq 0$, the total population will eventually decline. Thus, implementing controls is crucial to stabilize the population and mitigate the epidemic.
\section{Optimal Control Problem of the VS-EIAR Epidemic Model} \label{section 3}
\noindent

In this section, we propose an optimal control strategy to minimize the number of susceptible, exposed, infected, and asymptomatic individuals by implementing vaccination $v(t)$ for susceptible individuals and treatment $u(t)$ for infected individuals. The objective is to minimize the cost functional $\mathcal{J}$, defined on $U^{1}_{ad} \times U^{2}_{ad} \times \mathbb{R}^{*}_{+}$, as follows:
\begin{equation}
	\mathcal{J}(v, u, \tau) = \int_{0}^{\tau} \left[ \mathcal{K}\left(S(s), E(s), A(s), I(s)\right) + \frac{\sigma_0}{2} u(s)^2 + \sum_{i=1}^{n} \frac{\sigma_i \gamma_i^2}{2} v(s)^2 \right] ds + \mathcal{M}(\tau),
	\label{equ 2}
\end{equation}
where
\[
\mathcal{K}\left(S(t), E(t), A(t), I(t)\right) = \omega_1 S(t) + \omega_2 E(t) + \omega_3 A(t) + \omega_4 I(t) \quad \text{for } t \in [0, \tau],
\]
and $\mathcal{M}(\cdot)$ is a convex, non-negative, increasing continuous function satisfying $\lim\limits_{t \to +\infty} \mathcal{M}(t) = +\infty$.

The term $\mathcal{K}\left(S(t), E(t), A(t), I(t)\right)$ represents the \textit{epidemic cost} at time $t$, combining weighted contributions from susceptible ($S$), exposed ($E$), asymptomatic ($A$), and infected ($I$) individuals. The weights $\omega_i$ ($i = 1, 2, 3, 4$) reflect the relative importance or severity of each group in the epidemic. The term $\frac{\sigma_0}{2} u(t)^2$ represents the \textit{control cost} associated with treatment, where $u(t)$ is the treatment control and $\sigma_0$ is the controller gain, balancing treatment efficacy and cost. The term $\sum_{i=1}^{n} \frac{\sigma_i \gamma_i^2}{2} v(t)^2$ represents the \textit{control cost} associated with vaccination, where $v(t)$ is the vaccination control, $\sigma_i$ is the controller gain, and $\gamma_i$ is the effectiveness of the $i^{th}$ vaccination dose. Finally, $\mathcal{M}(\tau)$ represents the \textit{terminal cost}, incorporating additional costs or penalties related to the final state of the epidemic. The condition $\lim\limits_{t \to +\infty} \mathcal{M}(t) = +\infty$ reflects the increasing costs of disease control over time. The goal is to find an optimal control pair $(v^{*}, u^{*})$ and an optimal finite time $\tau^{*}$ such that
\[
\mathcal{J}(v^{*}, u^{*}, \tau^{*}) = \min \left\{ \mathcal{J}(v, u, \tau) \mid (v, u) \in U_{ad}^{1} \times U_{ad}^{2}, \tau \in \mathbb{R}^{+}_{*} \right\}.
\]

The proof of the following result follows directly from Theorem 23.11 in \cite{Clakre}. The uniqueness of the solution is guaranteed by the strict convexity of the cost functional $\mathcal{J}$.

\begin{theorem} 
	There exists a unique $(v^*, u^*, \tau^*) \in U^{1}_{ad} \times U^{2}_{ad} \times \mathbb{R}^{+}_{*}$ at which the cost functional $\mathcal{J}$ attains its minimum.
	\label{thm 3.1}
\end{theorem}

Let $X(t) = \left(S(t), E(t), A(t), I(t), R(t), D(t), V_1(t), \ldots, V_n(t)\right)^{T}$. Define the Hamiltonian function $H$ as:
\[
H(X(t), u(t), v(t), P(t), Q(t), t) = G(t) + [P(t), Q(t)]^{T} \dot{X}(t),
\]
where
\[
G(t) = \mathcal{K}\left(S(t), E(t), A(t), I(t)\right) + \frac{\sigma_0}{2} u(t)^2 + \left(\sum_{i=1}^{n} \frac{\sigma_i \gamma_i^2}{2}\right) v(t)^2,
\]
and
\[
P(t) = [p_1(t), p_2(t), \ldots, p_6(t)], \quad Q(t) = [q_1(t), q_2(t), \ldots, q_n(t)].
\]
The Hamiltonian can be expanded as:
\begin{eqnarray*}
	H &=& \omega_1 S(t) + \omega_2 E(t) + \omega_3 A(t) + \omega_4 I(t) + \frac{\sigma_0}{2} u(t)^2 + \left(\sum_{i=1}^{n} \frac{\sigma_i \gamma_i^2}{2}\right) v(t)^2 \\
	&& + p_1(t) \dot{S}(t) + p_2(t) \dot{E}(t) + p_3(t) \dot{A}(t) + p_4(t) \dot{I}(t) + p_5(t) \dot{R}(t) + p_6(t) \dot{D}(t) \\
	&& + \sum_{i=1}^{n} q_i(t) \dot{V}_i(t).
\end{eqnarray*}
The adjoint equations are given by:
\begin{eqnarray}
	\dot{P}(t) &=& -\left[ \frac{\partial H}{\partial S(t)}, \frac{\partial H}{\partial E(t)}, \frac{\partial H}{\partial A(t)}, \frac{\partial H}{\partial I(t)}, \frac{\partial H}{\partial R(t)}, \frac{\partial H}{\partial D(t)} \right],
	\label{eqn 4}
\end{eqnarray}
and
\begin{eqnarray}
	\dot{Q}(t) &=& -\left[ \frac{\partial H}{\partial V_1(t)}, \ldots, \frac{\partial H}{\partial V_n(t)} \right].
	\label{eqn 4-d}
\end{eqnarray}

The following theorem is the main result of this section, providing the explicit forms of the optimal controls $u^{*}$ and $v^{*}$. The proof is provided in the Appendix.
\begin{theorem} Let $(v^{*}, u^{*})$ be the optimal controls for equation \eqref{eq 1}, and let $S^{*}$, $E^{*}$, $A^{*}$, $I^{*}$, $R^{*}$, $D^{*}$, $V_1^{*}$, $\ldots$, $V_n^{*}$ denote the corresponding state values. Then, the optimal controls are given by:
	\begin{equation*}
	u^{*}(t)=\max\left\{\min\left\{\dfrac{I^{*}(t)\left[ p_4(t)-p_5(t)\right]}{\sigma_0},1 \right\},0\right\}, 
	\end{equation*}
	and
	\begin{equation*}
	v^{*}(t)=\max\left\{\min\left\{\dfrac{W(t)}{\left(\sum\limits_{i=1}^{n}\sigma_i\gamma_i^2\right)},\dfrac{1}{\gamma_1} \right\},0\right\},
	\end{equation*}
	where 
	\begin{equation}
	W(t)= \gamma_1S^{*}(t)[p_1(t)-q_1(t)]+\gamma_{2}q_1(t)V_1^*(t)+\sum\limits_{i=2}^{n-1}q_i(t)[\gamma_{i+1}V_i^*(t)-\gamma_{i}V^*_{i-1}(t)], 
	\label{W(t)}
	\end{equation}
	with $p_1$, $p_2$, $p_3$, $p_4$, $p_5$, $p_6$, $q_1$, $\ldots$, $q_n$ being the solutions of the adjoint equations:
	\begin{equation*}
	\left\{\begin{array}{l}
	\dot{p}_1(t)= \beta \Lambda^*(t)\left[p_1(t)-p_2(t)\right]+\gamma_1 v^*(t)\left[p_1(t)-q_1(t)\right]-\omega_1 \\[7pt]
	\dot{p}_2(t)= \beta\varepsilon S^*(t)p_1(t)+\left( k-\beta\varepsilon S^*(t)\right) p_2(t)-(1-z)k p_3(t)-zk p_4(t)-\omega_2\\[7pt]
	\dot{p}_3(t)= \beta \mu S^*(t)\left[p_1(t)-p_2(t) \right]+\eta p_3(t)-(1-p)\eta p_4(t)-\omega_3\\[7pt]
	\dot{p}_4(t)=\beta (1-q)S^*(t)\left[p_1(t)-p_2(t)\right]+ u^*(t)\left[p_4(t)-p_5(t)\right]+f(p_4(t)\\[7pt]
	\hspace{1.3cm}-\alpha p_5(t)) -\hspace{0.1cm} (1-\alpha)fp_6(t) -\omega_4\\[7pt]
	\dot{p}_5(t)= 0\\[7pt]
	\dot{p}_6(t)= 0\\[7pt]
	\dot{q}_1(t)=\delta_1\left[q_1(t)-p_2(t)\right] +\gamma_2v^*(t)\left[q_1(t)-q_2(t)\right] \\[7pt]
	\dot{q}_i(t)=-\delta_ip_2(t)+(\gamma_{i+1}v^*(t)+\delta_i)q_i(t)\quad(\text{for } i=2,3,\ldots,n-1)\\[7pt]
	\dot{q}_n(t)= 0,
	\end{array}\right.
	\end{equation*}
	for $t\in[0,\tau^*]$ with terminal conditions $p_j(\tau^*)=q_i(\tau^*)=0$ for $j=1,\ldots,6$ and $i=1,\ldots,n$. Here,  $\Lambda^*(t)=\varepsilon E^*(t)+(1-q)I^*(t)+\mu A^*(t)$ for $t\in[0,\tau^*]$.
	\label{thm 3.2}
\end{theorem}
\begin{remark}  The control objectives are to minimize the populations of susceptible ($S$), exposed ($E$), asymptomatic ($A$), and infected ($I$) individuals within an optimal finite time $\tau^{*}$ by applying vaccination and treatment strategies. The dynamics of the susceptible population are described by:
	\[
	\dot{S}(t) = -\Theta(t) S(t) \quad \text{for } t \geq 0,
	\]
	where $\Theta(t) = \beta \Lambda(t) + \gamma_1 v(t) > 0$ for $t \geq 0$. The positiveness of $\Theta(t)$ is guaranteed by the non-negativity of the states and parameters, and the control input $v(t)$ satisfies $0 < v(t) < \frac{1}{\gamma_1}$ when necessary. This implies that $S(t)$ decreases over time, and its solution is given by:
	\[
	S(t) = \exp\left(-\int_{0}^{t} \Theta(s) ds\right) S(0) \quad \text{for } t \geq 0.
	\]
	Since $S(0) > 0$, it follows that $S(t) \to 0$ asymptotically as $t \to +\infty$. From equation \eqref{eq 1}, the variation of $V_1(t)$ for $t\geq 0$ can be written as follows:
	\[\dot{V}_1(t)=-\Theta_1(t)V_1(t)+\gamma_1v(t)S(t) \quad\text{for } t\geq 0,\]
	where $\Theta_1(t)=\delta_1+\gamma_2v(t)>0$ for $t\geq 0$. Therefore,
	\begin{eqnarray*}
		V_1(t) &=& R_1(t,0)V_1(0)+\gamma_1\displaystyle\int_{0}^{t} R_1(t,s)v(s)S(s)ds\\
		&=& R_1(t,0)V_1(0)+\gamma_1\displaystyle\int_{0}^{t} R_1(t,t-s)v(t-s)S(t-s)ds\\
		&=& R_1(t,0)V_1(0)+\gamma_1\displaystyle\int_{0}^{+\infty}\chi_{[0,t]}(s) R_1(t,t-s)v(t-s)S(t-s)ds,
	\end{eqnarray*}
	where $R_1(t,s)=\exp\left( -\int_{s}^{t}\Theta_1(r)dr \right)$ for  $t\geq s\geq 0$. Since $V_1(0)>0$, $R_1(t,0)\to 0$, $\chi_{[0,t]}(s) R_1(t,t-s)v(t-s)S(t-s)\to 0$ as $t\to +\infty$ (because $v(\cdot)$ and $R_1(\cdot,\cdot)$ are bounded and $S(t)\to 0$ as $t\to +\infty$), and  $R_1(t,t-s)v(t-s)S(t-s)\leq \frac{S(0)e^{-\delta_{1}s}}{\gamma_1} $, we use dominated convergence theorem, we obtain that $V_1(t)\to 0$ as $t\to +\infty$. In a similar manner, we get that
	\begin{eqnarray*}
		V_i(t) 
		&=& R_i(t,0)V_i(0)+\gamma_i\displaystyle\int_{0}^{t} R_i(t,s)v(s)V_{i-1}(s)ds\\
		&=& R_i(t,0)V_i(0)+\gamma_i\displaystyle\int_{0}^{t} R_i(t,t-s)v(t-s)V_{i-1}(t-s)ds\\
		&=& R_i(t,0)V_i(0)+\gamma_i\displaystyle\int_{0}^{+\infty}\chi_{[0,t]}(s) R_i(t,t-s)v(t-s)V_{i-1}(t-s)ds,
	\end{eqnarray*}
	where
	\begin{center}
		$R_i(t,s)=\exp\left( -\displaystyle\int_{s}^{t}\Theta_i(r)dr \right)$ \text{ for }  $t\geq s\geq 0$,
	\end{center} 
	with $\Theta_i(t)=\delta_i+\gamma_{i+1}v(t)>0$ for $t\geq 0$, $i=2,\ldots,n-1$. Since $V_i(0)>0$, $R_i(t,0)\to 0$, $\chi_{[0,t]}(s)R_i(t,t-s)v(t-s)V_{i-1}(t-s)\to 0$ as $t\to +\infty$ (because $R_i(\cdot,\cdot)$ is bounded and $V_{i-1}(t)\to 0$ as $t\to +\infty$), and $\chi_{[0,t]}(s)R_i(t,t-s)v(t-s)V_{i-1}(t-s)\leq \frac{N_0e^{-\delta_{i}s}}{\gamma_1}$, it follows by the dominated convergence theorem that $V_i(t)\to 0$ as $t\to +\infty$ for $i=2,\ldots,n-1$. 
	Considering that 
	\begin{center}
		$\dot{V}_n(t)=\gamma_nv(t)V_{n-1}(t)$ \text{ for } $t\geq0$.
	\end{center}
	Since, $V_{n-1}(t)\to 0$ as $t\to +\infty$, it follows that $\dot{V}_n(t)\to 0$ as $t\to +\infty$ which means that $V_n(t)$ converges to its maximum over time. The variation of exposed population takes the following form:
	\begin{center}
		$\dot{E}(t)=\beta\Lambda(t)S(t)-kE(t)+\sum\limits_{i=1}^{n-1}\delta_iV_i(t)$ \text{ for } $t\geq0$.
	\end{center}
	Then, 
	\begin{center}
		$E(t)=e^{-kt}E(0)+\displaystyle\int_{0}^{t}e^{-ks}\left[\beta\Lambda(t-s)S(t-s)+\sum\limits_{i=1}^{n-1}\delta_iV_i(t-s)\right]ds$, \quad $t\geq 0$.
	\end{center}
	Since $E(0) > 0$, $S(t) \to 0$, and $V_i(t) \to 0$ as $t \to +\infty$, in a similar manner, we can show that $E(t) \to 0$ as $t \to +\infty$ thanks to the boundedness of $\Lambda(\cdot)$.
	For asymptomatic population, we have 
	\begin{center}
		$\dot{A}(t)=-\eta A(t)+(1-z)kE(t)$ \text{ for } $t\geq 0$,
	\end{center}
	which implies that 
	\begin{center}
		$A(t)=e^{-\eta t}A(0)+(1-z)k\displaystyle\int_{0}^{t}e^{-\eta s}E(t-s)ds$ \text{ for } $t\geq 0$.
	\end{center}
	Using the fact that $E(t)\to 0$ as $t\to +\infty$ and $A(0)> 0$, we show that $A(t)$ goes to $0$ as $t\to +\infty$. From equation \eqref{eq 1}, we have
	\begin{center}
		$\dot{I}(t)=-\Pi(t)I(t)+zkE(t)+(1-p)\eta A(t) $ \text{ for } $t\geq 0$,
	\end{center}
	where $\Pi(t)=f+u(t)>0$ for $t\geq 0$. Let 
	\begin{center}
		$\varPi(t,s)=\exp\left(-\displaystyle\int_{s}^{t}\Pi(r)dr\right) $ \text{ for } $t\geq s\geq 0$.
	\end{center}
	Then,
	\begin{eqnarray*}
		I(t)&=&\varPi(t,0)I(0)+\displaystyle\int_{0}^{t}\varPi(t,s)\left[ zkE(s)+(1-p)\eta A(s)\right]ds\\
		&=& \varPi(t,0)I(0)+\displaystyle\int_{0}^{t}\varPi(t,t-s)\left[ zkE(t-s)+(1-p)\eta A(t-s)\right]ds\\
		&=& \varPi(t,0)I(0)+\displaystyle\int_{0}^{+\infty}\hspace{-0.6cm}\chi_{[0,t]}(s)\varPi(t,t-s)\left[ zkE(t-s)+(1-p)\eta A(t-s)\right]ds.
	\end{eqnarray*}
	Since $\varPi(t,0)\to 0$, $E(t)\to 0$, $A(t)\to 0$ as $t\to +\infty$, and
	\begin{center}
		$\varPi(t,t-s)\left[ zkE(t-s)+(1-p)\eta A(t-s)\right]\leq \left[ zk+(1-p)\eta \right]N_0e^{-fs}$,
	\end{center}  
	by the dominated convergence theorem, we obtain that $I(t)\to 0$ as $t\to +\infty$. Using the fact that $I(t) \to 0$ as $t \to +\infty$, it follows that $\dot{D}(t) \to 0$ as $t \to +\infty$, implying that $D(t)$ converges to its maximum value. Similarly, for the recovered population, $\dot{R}(t) \to 0$ as $t \to +\infty$, since $(I(t), A(t)) \to (0, 0)$ as $t \to +\infty$. Consequently, $R(t)$ also converges to its maximum value as $t \to +\infty$. These results demonstrate that the control objectives are achieved: by vaccinating the population at rate $v(t)$ and treating infected individuals at rate $u(t)$, the spread of the disease can be eradicated.
	\label{rem 1}
\end{remark}

\section{An Impulsive VS-EIAR Epidemic Model} \label{section 4}
\noindent

This section introduces an impulsive VS-EIAR epidemic model that incorporates population immigration or travel. The model focuses on sudden additions to the susceptible, exposed, infected, and asymptomatic groups at specific times $t_k$ (where $t_k$ represents a particular day), with rates $\lambda_i(t)$ ($0 \leq \lambda_i(t) \leq 1$). For further details, refer to Figure \ref{Fig 2}. The dynamics of the controlled model are governed by the following system of ordinary differential equations (ODEs):

\begin{equation}
\left\{\begin{array}{l}\left.\begin{array}{l}
\dot{V}_1(t)= \gamma_1v(t)S(t)-\gamma_2v(t)V_1(t)-\delta_1V_1(t) \\[12pt]
\left\{\begin{array}{l}
\dot{V}_{i}(t)=\gamma_{i}v(t)V_{i-1}(t)-\gamma_{i+1}v(t)V_{i}(t)-\delta_{i}V_{i}(t),\\[12pt]
\text{for}\hspace{0.1cm} i=2,3,\ldots,n-1,
\end{array}\right.\\[22pt]
\dot{V}_n(t)=\gamma_nv(t)V_{n-1}(t)\\[12pt]
\dot{S}(t)=-(\beta\Lambda(t)+\gamma_1v(t))S(t)\\[12pt]
\dot{E}(t)=\beta\Lambda(t)S(t)-kE(t)+\sum\limits_{i=1}^{n-1}\delta_iV_i(t)\\[12pt]
\dot{A}(t)=(1-z)kE(t)-\eta A(t)\\[12pt]
\dot{I}(t)=zkE(t)+(1-p)\eta A(t)-fI(t)-u(t)I(t)\\[12pt]
\dot{R}(t)=\alpha fI(t)+u(t)I(t)+p\eta A(t)\\[12pt]
\dot{D}(t)= (1-\alpha)fI(t)\\[12pt]
\end{array}\right\}{\begin{array}{l}
	\hspace{0.1cm}t\in[0,\tau], \hspace{0.1cm}t\neq t_k\\[12pt]
	\hspace{0.1cm} k=1,2,\ldots,p.\\[12pt]
	\hspace{0.1cm} \text{where}\hspace{0.1cm} p\in\mathbb{N}^{*}.
	\end{array}}\\[12pt]
\dot{V}_1(t^{+}_k)= \gamma_1v(t_k)S(t_k)-\gamma_2v(t_k)V_1(t_k)-\delta_1V_1(t_k) \\[12pt]
\left\{\begin{array}{l}
\dot{V}_{i}(t^{+}_k)=\gamma_{i}v(t_k)V_{i-1}(t_k)-\gamma_{i+1}v(t_k)V_{i}(t_k)-\delta_{i}V_{i}(t_k),\\[12pt]
\text{for}\hspace{0.1cm} i=2,3,\ldots,n-1,
\end{array}\right.\\[22pt]
\dot{V}_n(t^{+}_k)=\gamma_nv(t_k)V_{n-1}(t_k)\\[12pt]
\dot{S}(t^{+}_k)=-(\beta\Lambda(t_k)+\gamma_1v(t_k))S(t_k)+\lambda_1(t_k)S(t_k)\\[12pt]
\dot{E}(t^{+}_k)=\beta\Lambda(t_k)S(t_k)-kE(t_k)+\lambda_2(t_k)E(t_k)+\sum\limits_{i=1}^{n-1}\delta_iV_i(t_k)\\[12pt]
\dot{A}(t^{+}_k)=(1-z)kE(t_k)-\eta A(t_k)+\lambda_3(t_k)A(t_k)\\[12pt]
\dot{I}(t^{+}_k)=zkE(t_k)+(1-p)\eta A(t_k)-fI(t_k)+\lambda_4(t_k)I(t_k)-u(t_k)I(t_k)\\[12pt]
\dot{R}(t^{+}_k)=\alpha fI(t_k)+u(t_k)I(t_k)+p\eta A(t_k)\\[12pt]
\dot{D}(t^{+}_k)= (1-\alpha)fI(t_k).
\end{array}\right.
\label{equ 3}
\end{equation}
\begin{figure}[ht]
	\centering
	\scalebox{0.9}{
		\begin{tikzpicture}[node distance=2.9cm, auto, thick, >=stealth]
			
			\node [ellipse, draw, fill=green!60, align=center, minimum width=1.4cm, minimum height=1.4cm] (R) {R};
			
			\node [ellipse, draw, fill=red!60, below of=R, align=center, minimum width=1.4cm, minimum height=1.4cm] (I) {I};
			
			\node [ellipse, draw, fill=red!30, left of=I, align=center, minimum width=1.4cm, minimum height=1.4cm] (A) {A};
			
			\node [ellipse, draw, fill=black!10, right of=I, align=center, minimum width=1.4cm, minimum height=1.4cm] (D) {D};
			
			\node [ellipse, draw, fill=yellow!60, below of=I, align=center, minimum width=1.4cm, minimum height=1.4cm] (E) {E};
			
			\node [ellipse, draw, fill=green!50, below of=E, align=center, minimum width=1.4cm, minimum height=1.4cm] (V2) {\(V_2\)};
			
			\node [ellipse, draw, fill=green!30, left of=V2, align=center, minimum width=1.4cm, minimum height=1.4cm] (V1) {\(V_1\)};
			
			\node [ellipse, draw, fill=blue!30, left of=V1, align=center, minimum width=1.4cm, minimum height=1.4cm] (S) {S};
			
			\node [ellipse, draw, fill=green!65, right of=V2, align=center, minimum width=1.4cm, minimum height=1.4cm] (Vn1) {\(V_{n-1}\)};
			
			\node [ellipse, draw, fill=green!80, right of=Vn1, align=center, minimum width=1.4cm, minimum height=1.4cm] (Vn) {\(V_n\)};
			
			\draw[->, line width=0.7pt, >=stealth, shorten <=0.5mm, shorten >=0.5mm] (S) -- (E) node[midway, sloped, above] {\(\scriptstyle \beta\Lambda\)};
			
			\draw[->, line width=0.7pt, >=stealth, shorten <=0.5mm, shorten >=0.5mm] (V1) -- (E) node[midway, sloped, above] {\(\scriptstyle \delta_1\)};
			
			\draw[->, line width=0.7pt, >=stealth, shorten <=0.5mm, shorten >=0.5mm] (V2) -- (E) node[midway, left] {\(\scriptstyle \delta_2\)};
			
			\draw[->, line width=0.7pt, >=stealth, shorten <=0.5mm, shorten >=0.5mm] (Vn1) -- (E) node[midway, sloped, above] {\(\scriptstyle \delta_{n-1}\)};
			
			\draw[->, line width=0.7pt, >=stealth, shorten <=0.5mm, shorten >=0.5mm] (Vn) -- (E) node[midway, sloped, above] {\(\scriptstyle \delta_{n}\)};
			
			\draw[->, line width=0.8pt,draw=green!50!black, >=stealth, shorten <=0.5mm, shorten >=0.5mm] (S) -- (V1) node[midway, sloped, below] {\(\scriptstyle \gamma_1v\)};
			
			\draw[->, line width=0.8pt,draw=green!50!black, >=stealth, shorten <=0.5mm, shorten >=0.5mm] (V1) -- (V2) node[midway, sloped, below] {\(\scriptstyle \gamma_2v\)};
			
			\draw[dashed,-, line width=0.7pt,draw=green!50!black, >=stealth, shorten <=0.5mm, shorten >=0.5mm] (V2) -- (Vn1) node[midway, sloped, below] {};
			
			\draw[->, line width=0.8pt,draw=green!50!black, >=stealth, shorten <=0.5mm, shorten >=0.5mm] (Vn1) -- (Vn) node[midway, sloped, below] {\(\scriptstyle \gamma_nv\)};
			
			\draw[->, line width=0.7pt, >=stealth, shorten <=0.5mm, shorten >=0.5mm] (E) -- (A) node[midway, sloped, above] {\(\scriptstyle (1-z)k\)};
			
			\draw[->, line width=0.7pt, >=stealth, shorten <=0.5mm, shorten >=0.5mm] (E) -- (I) node[midway, right] {\(\scriptstyle zk\)};
			
			\draw[->, line width=0.7pt, >=stealth, shorten <=0.5mm, shorten >=0.5mm] (A) -- (I) node[midway, sloped, above] {\(\scriptstyle (1-p)\eta\)};
			
			\draw[->, line width=0.7pt, >=stealth, shorten <=0.5mm, shorten >=0.5mm] (I) -- (D) node[midway, sloped, above] {\(\scriptstyle (1-\alpha)f\)};
			
			\draw[->, line width=0.7pt, >=stealth, shorten <=0.5mm, shorten >=0.5mm] (A) -- (R) node[midway, sloped, above] {\(\scriptstyle p\eta\)};
			
			\draw[->, line width=0.7pt, >=stealth, shorten <=0.5mm, shorten >=0.5mm] (I) -- (R) node[midway, left] {\(\scriptstyle \alpha f\)};
			
			\draw[->, draw=green!50!black, line width=0.7pt, shorten <=0.5mm, shorten >=0.5mm, bend right=45] (I) to node[midway, right] {\(\scriptstyle u\)} (R);
			
			\draw[<-, line width=0.6pt, shorten <=0.5mm, shorten >=0.5mm, draw=red] (S) -- ++(0,1.5) node[midway, left] {\(\scriptstyle \lambda_1(t_k)\)};
			
			\draw[<-, line width=0.6pt, shorten <=0.5mm, shorten >=0.5mm, draw=red] (E) -- ++(-2,0) node[midway, sloped, above] {\(\scriptstyle \lambda_2(t_k)\)};
			
			\draw[<-, line width=0.6pt, shorten <=0.5mm, shorten >=0.5mm, draw=red] (I) -- ++(1.5,-1.5) node[midway, sloped, above] {\(\scriptstyle \lambda_4(t_k)\)};
			
			\draw[<-, line width=0.6pt, shorten <=0.5mm, shorten >=0.5mm, draw=red] (A) -- ++(0,1.5) node[midway, left] {\(\scriptstyle \lambda_3(t_k)\)};
		\end{tikzpicture}
	}
\caption{The VS-EIAR epidemic model incorporating impulsive growth dynamics.}
	\label{Fig 2}
\end{figure}
The following theorem guarantees the existence, uniqueness, positivity, and boundedness of solutions for equation \eqref{equ 3}. The proof is provided in the Appendix.
\begin{theorem}
	Let $V_{1,0}\geq 0$, $\ldots$, $V_{n,0}\geq 0$, $S_0\geq 0$, $E_0\geq 0$, $A_0\geq 0$, $I_0\geq 0$, $R_0\geq 0$, and $D_0\geq 0$. Let $v\in U^{1}_{ad}$ and $u\in U^{2}_{ad}$ be fixed. Then, equation $\eqref{equ 3}$  has a unique bounded positive solution defined on $\mathbb{R}^{+}$.
	\label{thm 4}
\end{theorem}

\section{Optimal Control Problem of the Impulsive VS-EIAR Epidemic Model} \label{section 5}
\noindent

Following the approach in Section \ref{section 3}, we aim to minimize the cost functional $\mathcal{J}$ defined in equation \eqref{equ 2}. The existence of a minimum for $\mathcal{J}$ is also guaranteed in the impulsive case. We use Theorem 23.11 in \cite{Clakre}. The following theorem is the main result of this section, and its proof follows similarly to that of Theorem \ref{thm 3.1}.

\begin{theorem} Let $(u^*,v^*)$ be the optimal controls of equation $(\ref{equ 3})$ and $S^{*}$, $E^{*}$, $A^{*}$, $I^{*}$, $R^{*}$, $D^{*}$, $V_1^{*}$,$\ldots$,$V_n^{*}$ be the states values corresponding to $(u^*,v^*)$. Then,
	\begin{equation*}
	u^{*}(t)=\max\left\{\min\left\{\dfrac{I^{*}(t)\left[ p_4(t)-p_5(t)\right]}{\sigma_0},1 \right\},0\right\}, 
	\end{equation*}
	and
	\begin{equation*}
	v^{*}(t)=\max\left\{\min\left\{\dfrac{W(t)}{\left(\sum\limits_{i=1}^{n}\sigma_i\gamma_i^2\right)},\dfrac{1}{\gamma_1} \right\},0\right\},
	\end{equation*}
	where $W(t)$ is given by \eqref{W(t)}, and  $p_1$, $p_2$, $p_3$, $p_4$, $p_5$, $p_6$, $q_1$, $\ldots$, $q_n$ being the solutions of the following equation:
	\begin{equation*}
	\left\{\begin{array}{l}
	\dot{p}_1(t)= \beta \Lambda^*(t)\left[p_1(t)-p_2(t)\right]+\gamma_1 v^*(t)\left[p_1(t)-q_1(t)\right]-\varrho_1(t)-\omega_1 \\[12pt]
	\dot{p}_2(t)= \beta\varepsilon S^*(t)[p_1(t)-p_2(t)]+ k \left[ p_2(t)-(1-z) p_3(t)-z p_4(t)\right]-\varrho_2(t)-\omega_2\\[12pt]
	\dot{p}_3(t)= \beta \mu S^*(t)\left[p_1(t)-p_2(t) \right]+\eta p_3(t)-(1-p)\eta p_4(t)-\varrho_3(t)-\omega_3\\[12pt]
	\dot{p}_4(t)=\beta (1-q)S^*(t)\left[p_1(t)-p_2(t)\right]+ u^*(t)\left[p_4(t)-p_5(t)\right]+f(p_4(t)-\alpha p_5(t))\\[12pt]
	\hspace{1.3cm} - \hspace{0.1cm} (1-\alpha)fp_6(t) -\varrho_4(t) -\omega_4\\[12pt]
	\dot{p}_5(t)=0\\[12pt]
	\dot{p}_6(t)=0\\[12pt]
	\dot{q}_1(t)=\delta_1\left[q_1(t)-p_2(t)\right] +\gamma_2v^*(t)\left[q_1(t)-q_2(t)\right] \\[12pt]
	\dot{q}_i(t)=-\delta_ip_2(t)+(\gamma_{i+1}v^*(t)+\delta_i)q_i(t)\quad(\text{for } i=2,3,\ldots,n-1),\\[12pt]
	\dot{q}_n(t)=0,
	\end{array}\right.
	\end{equation*}
	for $t\in[0,\tau^*]$ with terminal conditions $p_j(\tau^*)=q_i(\tau^*)=0$ for $j=1,\ldots,6$ and $i=1,\ldots,n$. Here,  $\Lambda^*(\cdot)=\varepsilon E^*(\cdot)+(1-q)I^*(\cdot)+\mu A^*(\cdot)$, and $\varrho_l(t)=\lambda_l(t)$ if $t=t^{+}_k$ ($\varrho_l(t)=0$ if $t\neq t^{+}_k$), for $k=1,\ldots,p$ and $l=1,\ldots,4$.
	\label{theorem 5.1}
\end{theorem}

\begin{remark} Considering that $\dot{S}(t)=- (\Theta(t)-\varrho_1(t))S(t)$ for $t\geq0$, where $\Theta(t)=\beta\Lambda(t)+\gamma_1v(t)>0$ for $t\geq 0$ and $\varrho_1(t)$ is defined as in Theorem \ref{theorem 5.1}. Then,
	\begin{center}
		$S(t)=\exp\left(-\displaystyle\int_{0}^{t}\left[\Theta(s)-\varrho_1(s)\right]ds\right)S(0)$ \text{ for } $t\geq 0$.
	\end{center}
	Since $0\leq \varrho_1(t)\leq 1$, and $v(t)$ is the control input and when needed, we can choose $v(t)=\frac{1}{\gamma_1}$, it follows that $\Theta(t)>\varrho_1(t)$ given that the parameters and states are positives. Thus, since $S(0)> 0$, we show that  $S(t)\to 0$ as $t\to+\infty$. According to remark \ref{eq 1}, in a similar manner, we can prove that $V_i(t)$ goes to $0$ for $i=1,\ldots,n-1$, and $V_n(t)$ converges to its maximum as $t\to +\infty$. For exposed individuals, we have
	\begin{center}
		$\dot{E}(t)=-(k-\varrho_2(t))E(t)+\beta\Lambda(t)S(t)+\sum\limits_{i=1}^{n-1}\delta_iV_i(t)$ \text{ for } $t\geq 0$,
	\end{center}
	where $\varrho_2(\cdot)$ is defined as in Theorem \ref{theorem 5.1}. Let 
	\begin{center}
		$\Upsilon_1(t,s)=\exp\left(-\displaystyle\int_{s}^{t}(k-\varrho_2(r))dr\right)$ \text{ for } $t\geq s\geq 0$.
	\end{center} 
	Then,
	\begin{center}
		$E(t)=\Upsilon_1(t,0)E(0)+\displaystyle\int_{0}^{t}\Upsilon_1(t,t-s)\left[\beta\Lambda(t-s)S(t-s)+\sum\limits_{i=1}^{n-1}\delta_iV_i(t-s) \right]ds$ \text{ for } $t\geq 0$.
	\end{center}
	Furthermore, for $t\geq s\geq 0$, we have
	\begin{eqnarray*}
		\Upsilon_1(t,s)&=&\exp\left(-k(t-s)+\displaystyle\int_{s}^{t}\varrho_2(r)dr\right)\\
		&=& \exp\left(-k(t-s)+\sum\limits_{s\leq t_k\leq  t}\lambda_2(t_k)\right)\\
		& \leq & \exp\left(-k(t-s)+\sum\limits_{k=1}^{p}\lambda_2(t_k)\right).
	\end{eqnarray*}
	Moreover, $$\Upsilon_1(t,t-s)\leq \exp\left(-ks+\sum\limits_{k=1}^{p}\lambda_2(t_k)\right) \text{ for } t\geq s\geq 0.$$ Since $S(t)\to 0$, and $V_i(t)\to 0$ (for $i=1,\ldots,n-1$) as $t\to +\infty$, it follows that $E(t)\to 0$ as $t\to +\infty$ provided that $E(0)>0$.
	Considering that
	\begin{center}
		$\dot{A}(t)=-(\eta-\varrho_3(t)) A(t)+(1-z)kE(t)$ \text{ for } $t\geq 0$
	\end{center}
	where $\varrho_3(\cdot)$ is defined as in Theorem \ref{theorem 5.1}. Let 
		$\Upsilon_2(t,s)=\exp\left(-\int_{s}^{t}(\eta-\varrho_3(r))dr\right)$ for $t\geq s\geq 0$.
	Then,
	\begin{center}
		$A(t)=\Upsilon_2(t,0)A(0)+(1-z)k\displaystyle\int_{0}^{t}\Upsilon_2(t,t-s)E(t-s)ds$ \text{ for } $t\geq 0$.
	\end{center}
	As previously, we can show that 
	\begin{center}
		$\Upsilon_2(t,s)\leq \exp\left(-\eta(t-s)+\sum\limits_{k=1}^{p}\lambda_3(t_k)\right)$ \text{ for } $t\geq s\geq 0$,
	\end{center}
	which implies that, $$\Upsilon_2(t,t-s)\leq \exp\left(-\eta s+\sum\limits_{k=1}^{p}\lambda_3(t_k)\right) \text{ for } t\geq s\geq 0.$$ Since $A(0)> 0$, $\Upsilon_2(t,0)\to 0$, and $E(t)\to 0$ as $t\to +\infty$, it follows that $A(t)\to 0$ as $t\to +\infty$.
	For infected population, we have 
	\begin{center}
		$\dot{I}(t)=-(f+u(t)-\varrho_4(t))I(t)+zkE(t)+(1-p)\eta A(t)$ \text{ for } $t\geq 0$,
	\end{center} 
	where $\varrho_4(\cdot)$ is defined as in Theorem \ref{theorem 5.1}. Let
	\begin{center}
		$\Upsilon_3(t,s)=\exp\left(-\displaystyle\int_{s}^{t}(f+u(r)-\varrho_4(r))dr\right)$ \text{ for } $t\geq s\geq0$,
	\end{center}
	then
	\begin{center}
		$I(t)=\Upsilon_3(t,0)I(0)+\displaystyle\int_{0}^{t}\Upsilon_3(t,t-s)\left[zkE(t-s)+(1-p)\eta A(t-s)\right]ds$  \text{ for } $t\geq 0$.
	\end{center}
	Since $0\leq \lambda_4(t)\leq 1$, and $u(t)$ is the control input and when is needed, we can choose $u(t)=1$, it follows that $f+u(t)>\varrho_4(t)$, which implies that $\Upsilon_3(t,0)\to 0$ as $t\to +\infty$. Using the fact that $E(t)\to 0$, and $A(t)\to 0$ when $t$ goes to infinity, we prove that $I(t)\to 0$ as $t\to +\infty$. For the deceased population, we observe that $\dot{D}(t) \to +\infty$ as $t \to +\infty$, which implies that $D(t)$ converges to its maximum value. Additionally, since $E(t)$, $A(t)$, and $I(t)$ decrease to zero as $t$ goes to infinity, we find that $\dot{R}(t) \to 0$ as $t \to +\infty$, which implies that $R(t)$ converges to its maximum as $t$ goes infinity. Consequently, the control objectives for the impulsive model are achieved.
	\label{rem 4}
\end{remark}
\begin{remark} 
	In the impulsive case, additional individuals join the susceptible, exposed, asymptomatic, and infected groups at rates $\lambda_i(t)$. These individuals contribute to the spread of the pandemic, ensuring that the populations of these groups remain non-zero for an extended period. However, the controller can still achieve the eradication of disease propagation.
\end{remark}
\section{Application to COVID-19} \label{section 6}
\noindent

The parameters used in this study are derived from \cite{Li} and \cite{Riou}, with some values adjusted to reflect the specific characteristics of COVID-19. The parameters and initial states are summarized in Table \ref{table1_Num}. We assume that susceptible individuals receive two doses of vaccination, i.e., $n = 2$. Figures \ref{Fig 4}--\ref{Fig 10} and \ref{Fig 14}--\ref{Fig 20} illustrate the variation in the number of individuals across each group, comparing scenarios with and without controls for both models. The red curves represent the uncontrolled cases, while the green curves represent the controlled cases.

\begin{table}[h!]
	\centering
	\caption{Initial states and Model parameters}\label{table1_Num}
	\begin{tabular}{ccc} 
		\hline
		Initial States & Value \\
		\hline \\[-8pt]
		$S_0$ & $8\times 10^{3}$\\ 
		$E_0$ & $1\times 10^{3}$\\ 
		$A_0$ & $5\times 10^{2}$\\ 
		$I_0$ & $5\times 10^{2}$\\ 
		$R_0$, $D_0$ & 0\\ 
		$V_{1,0}$, $V_{2,0}$  & 0\\
		\hline
	\end{tabular}
	\hspace{0.3cm}
	\begin{tabular}{ccc} 
		\hline
		Parameters & Value \\
		\hline \\[-8pt]
		$p$ & 0.02\\
		$q$ & 0.5\\
		$\eta$ & 0.3\\
		$z$ & 0.1\\ 
		$\varepsilon$ & 0\\
		$\mu$ & 1\\  
		\hline
	\end{tabular}
	\hspace{0.3cm}
	\begin{tabular}{ccc} 
		\hline
		Parameters & Value \\
		\hline \\[-8pt]
		$\delta_1$ & $5\times 10^{-4}$\\
		$\alpha$ & 0.995\\
		$k$ & 0.54\\
		$f$  & 0.3\\
		$\gamma_1$ & 1\\ 
		$\gamma_2$ & 1\\ 
		\hline
	\end{tabular}
\end{table}

\subsection{The VS-EIAR Epidemic Model}
\noindent

Figure \ref{Fig 12} represents the evolution of vaccinated individuals over time for the VS-EIAR epidemic model. It is clear that the population with two doses of vaccine comprises almost $7800$ persons, which is approximately $78\%$ of the entire population. Figure \ref{Fig 4} depicts the development of susceptible individuals with and without controls over a $35$-day period. As demonstrated, the susceptible population reaches zero in around five days when controls are present. In contrast, when controls are absent, the susceptible population either never reaches zero, or it takes a longer period of time to reach zero. Note that while controls are absent, nearly all susceptible individuals fall into the group of exposed population (see Figure \ref{Fig 4}), whereas in the presence of controls, almost $78\%$ of them fall into the group of vaccinated individuals (see Figure \ref{Fig 12}). Figure \ref{Fig 4} illustrates the changes in exposed individuals with and without controls over time. Be aware that in the absence of controls, the number of exposed population takes a long time (more than $35$ days) to reach zero. This is typical given that, during this time period, the population is moving in from the susceptible group (see Figure \ref{Fig 4}). However, when controls are present, the number decreases to zero in around $12$ days. This is because a few persons have been pulled from the susceptible population as a result of the applied controls. Figure \ref{Fig 6} shows the progression of both the asymptomatic population with and without controls. In the absence of controls, the number of asymptomatic individuals increases widely from the first day to the $20th$ day, reaching roughly $1350$ persons, which is a significant quantity when compared to the total population (around $10000$ persons). This number decreases to around $800$ individuals in the presence of controls. Additionally, when controls are in place, the number of individuals in this group tends to decrease rapidly (within about $25$ days), which is not the case when controls are absent. Figure \ref{Fig 6} compares the number of infected population over time with and without controls. It is obvious that the number of infected individuals significantly increases from the first day to the $25th$ day when controls are absent, going from $500$ on the first day to almost $1350$ persons, a statistic not suggested when compared to the total population. By contrast, in the presence of controls, it nearly disappears to nil in around $20$ days. Figure \ref{Fig 8} represents the number of recovered individuals from the virus after $35$ days, both with and without a controller. We may observe that the number of recovered individuals decreases in the presence of controls, while it increases in the absence of controls. When controls are in place, we recover just approximately $200$ individuals, compared to when controls aren't applied, where we recover almost $95\%$ of the total population. This explains the fact that when susceptible individuals are vaccinated against the infection, fewer individuals will get infected, and consequently, fewer individuals will recover from it (see also Figure \ref{Fig 12}). Figure \ref{Fig 8} depicts the progression of deaths because of the infection both with and without controls. The number of deaths with controls cannot exceed $3$ persons per $10000$ persons, which is more acceptable. By contrast, in the absence of controls, the number of deaths keeps increasing to reach more than $44$ persons during $35$ days, which is unnatural. As a result, the availability of vaccines allows us to reduce the number of infected individuals and, ultimately, the number of deaths. Figure \ref{Fig 10} shows that in the absence of controls, a higher percentage of the population dies (about $0.44\%$ of the total population die from their infection during $35$ days). By contrast, with controls, more of the population stays alive, approximately $9997$ (around $ 10000$). The reason the number of persons spared from the virus is not exactly the same as the total population is that $0.03\%$ of the population dies because of the infection (more of them are not vaccinated). Figure \ref{Fig 10} represents the absolute difference between the total populations $N(t)$ both with and without controls. We can see that in the presence of controls, we can save about $41$ persons from deaths during $35$ days. Figure \ref{Fig 12} depicts two cases: recovered individuals with and without controls. It is clear that in the absence of controls, the recovered population from the virus is approximately equal to the number of the total population, which means that all population is infected by the disease. A part of them dies because of the infection. By contrast, if we vaccinate the susceptible individuals, we can save about $78\%$ of the total population from the infection, while recovering $21.97\%$ from it.
\subsection{The impulsive VS-EIAR Epidemic Model}
\noindent

Figure \ref{Fig 22} illustrates the changes in the vaccinated population with the impulsive rate of growth. According to this data, over $8000$ susceptible individuals have received vaccinations during a $35$-day period. From Figure \ref{Fig 14}, we can notice that the additional population can increase the number of susceptible individuals in this group. However, if we put controls in place, we can eliminate them completely within ten days. Figure \ref{Fig 14} provides information about the exposed population with the impulsive rate of growth during $35$ days. It is evident that the number of exposed individuals grows as a result of the additional population, but in the presence of controls, we can reduce that number to zero, whereas in the absence of controls, we cannot. Figure \ref{Fig 16} shows that the additional population initially increases the number of asymptomatic individuals before starting to decrease, but this process takes more time, leading to more infections. However, with controls in place, we can rapidly eradicate infections and bring the number of asymptomatic individuals down to zero. According to Figure \ref{Fig 16}, we can notice that in the absence of controls, there is a significant increase in the number of infected individuals starting from the first day to the $20th$ day, primarily due to the additional population. Additionally, the infected population will never go extinct, at least not for a long time. However, with controls in place, the number of infected individuals decreases rapidly, eventually reaching zero. This indicates the effectiveness of controls in eradicating the infection, resulting in fewer deaths and recoveries, as shown in Figures \ref{Fig 18} and \ref{Fig 18}. Figure \ref{Fig 18} compares the number of recovered individuals from the virus in the absence and in the presence of controls. It is clear that without controls, we can recover more than $9600$ persons from the disease, which is roughly the entire population. Contrarily, when controls are present, we can only recover around $3000$ population, the majority of whom are unvaccinated or have incomplete vaccination. Figure \ref{Fig 18} illustrates the evolution of the deaths population with the impulsive rate of growth, both with and without controls, showing that in the absence of controls, there are many deaths, with almost $45$ persons succumbing to the infection during $35$ days. This rate is concerning for a community of no more than $12,000$ persons. However, when controls are implemented, the percentage of deaths significantly reduces to $0.03\%$, which is more acceptable. The difference between the changes in the overall population size under the proposed controls and those that occur when there are no controls is shown in Figure \ref{Fig 20}. As we can see, we can keep more of the population alive in the presence of controls, while we lose a part of them in the absence of controls. Additionally, it is noticeable that the population is not constant since the addition of the new population. In Figure \ref{Fig 20}, the difference between the population with the impulsive rate of growth in the presence and absence of controls is demonstrated. Significantly, with the implementation of controls, we are able to maintain over 44 individuals per month, which carries important implications for a population not exceeding 12,000 people. This demonstrates the effectiveness of controls in ensuring a more stable and sustainable population, unlike the scenario without controls, where population growth is less regulated, leading to potential fluctuations. Figure \ref{Fig 22} provides a summary of the development of both recovered and vaccinated individuals. The data clearly indicates that immunizing a larger portion of the population against the virus leads to saving more lives. With increasing vaccination rates, the number of recovered individuals also rises, as a significant portion of the population becomes immune to the virus. This emphasizes the importance of widespread vaccination efforts in curbing the impact of the virus on the population and reducing the overall burden on healthcare systems.

\begin{figure}[ht]
	\centering
	\includegraphics[scale=0.86]{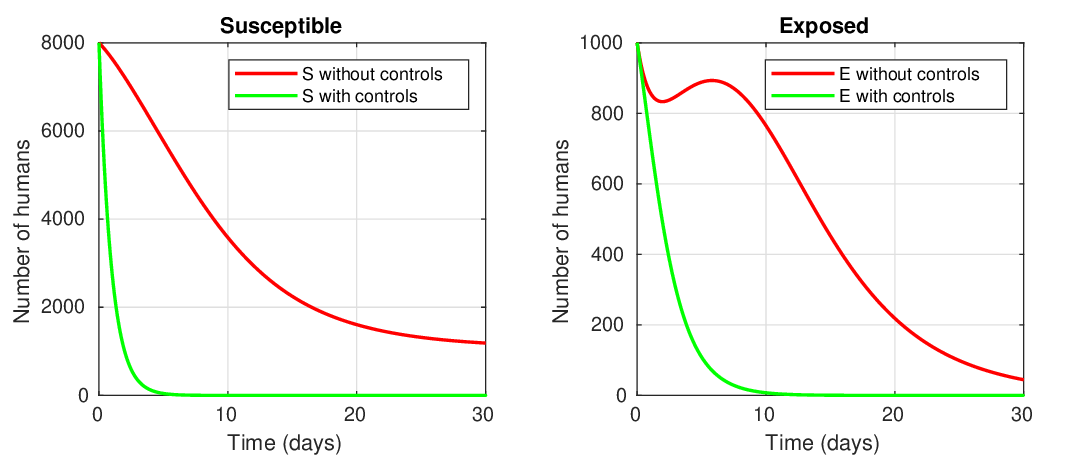} 
	\caption{Changes in the Susceptible and Exposed groups without the impulsive rate of growth.}
	\label{Fig 4}
\end{figure}

\begin{figure}[ht]
	\centering
	\includegraphics[scale=0.86]{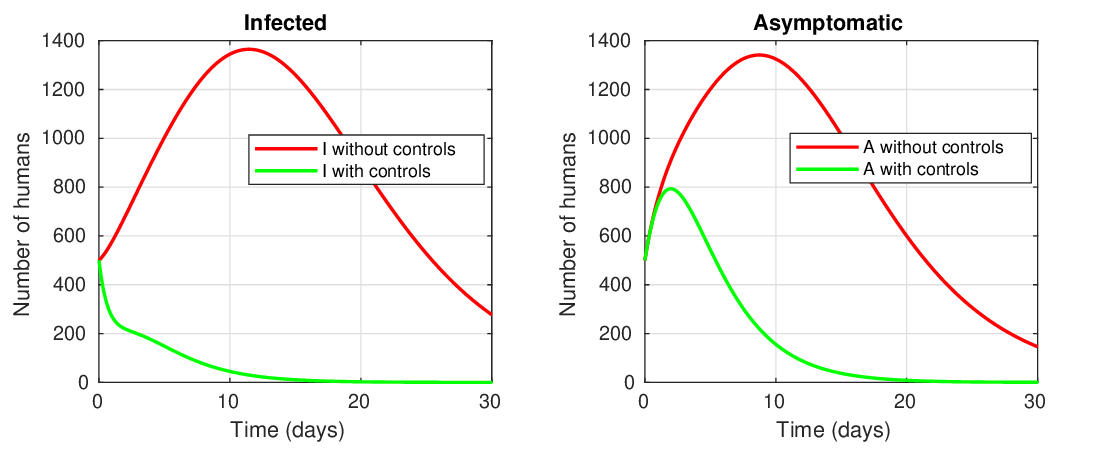} 
	\caption{Changes in the Asymptomatic and Infected groups without the impulsive rate of growth.}
	\label{Fig 6}
\end{figure}

\begin{figure}[ht]
	\centering
	\includegraphics[scale=0.86]{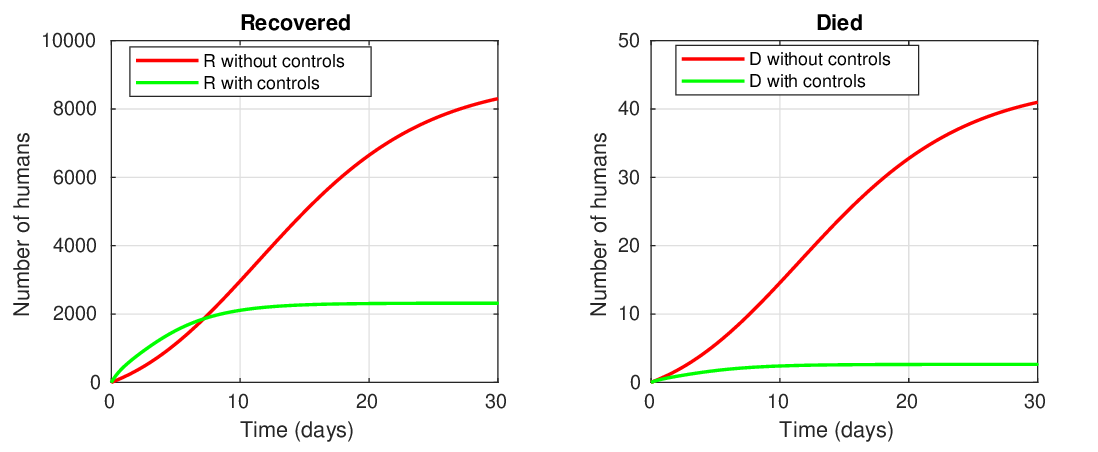} 
	\caption{Changes in the Recovered and Death groups  without the impulsive rate of growth.}
	\label{Fig 8}
\end{figure}

\begin{figure}[ht]
	\centering
	\includegraphics[scale=0.86]{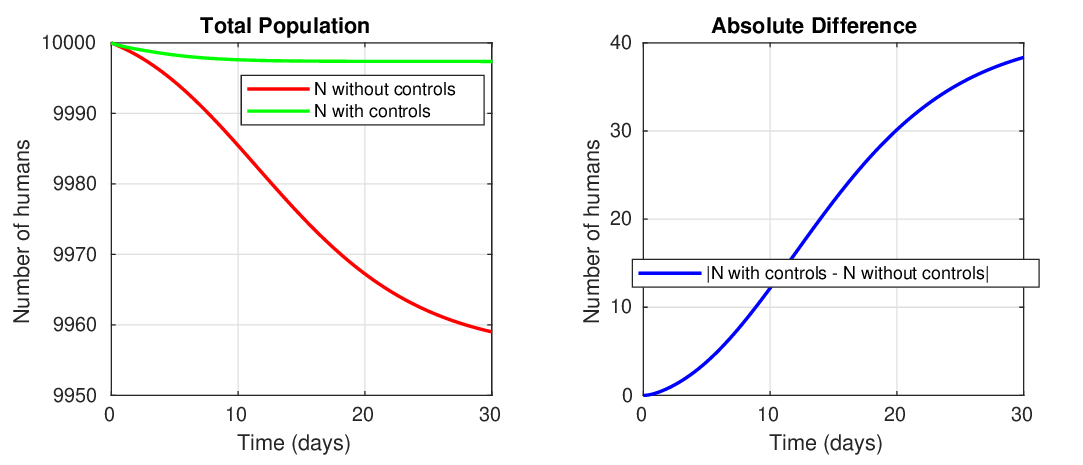} 
	\caption{Changes in the total population without the impulsive rate of growth.}
	\label{Fig 10}
\end{figure}

\begin{figure}[ht]
	\centering
	\includegraphics[scale=0.86]{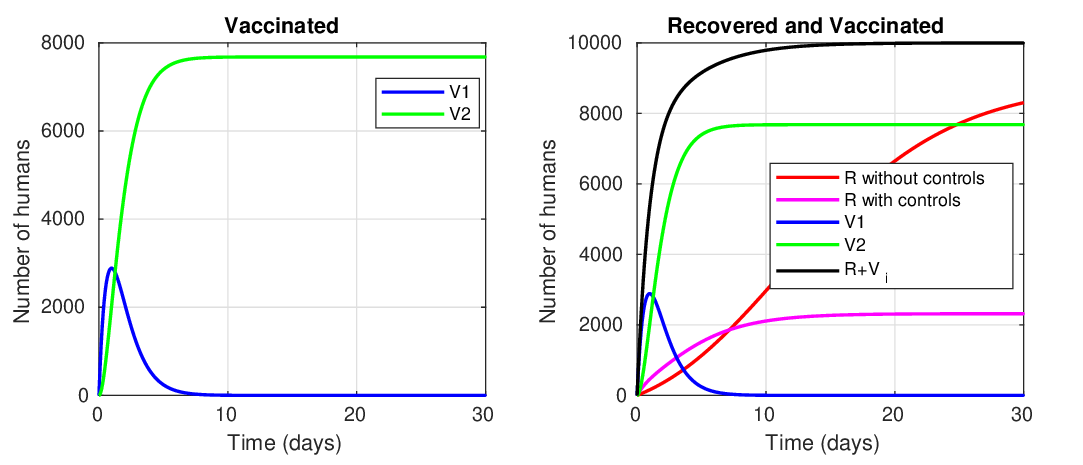} 
	\caption{Comparison between Recovered and Vaccinated population and a total of them with and without controls.}
	\label{Fig 12}
\end{figure}

\begin{figure}[ht]
	\centering
	\includegraphics[scale=0.86]{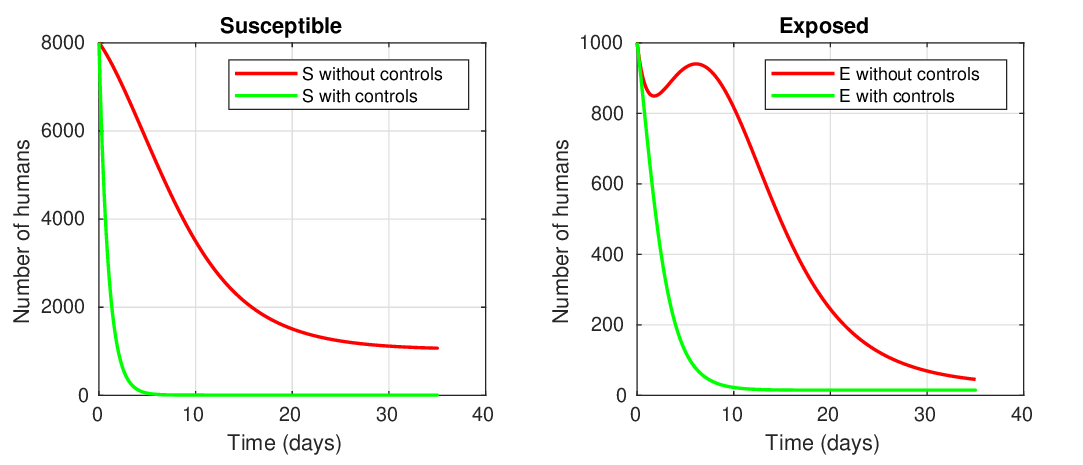} 
	\caption{Changes in Susceptible and Exposed groups with the impulsive rate of growth.}
	\label{Fig 14}
\end{figure}

\begin{figure}[ht]
	\centering
	\includegraphics[scale=0.86]{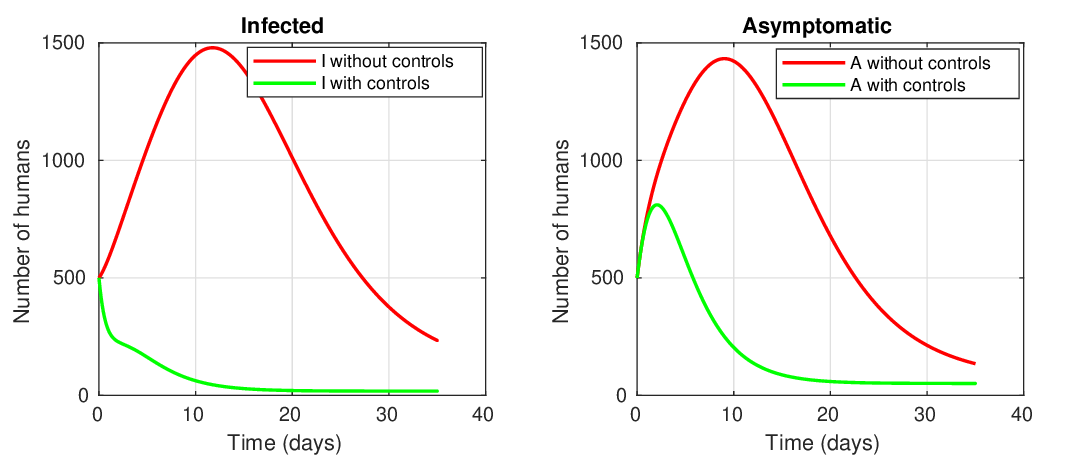} 
	\caption{Changes in the Asymptomatic and Infected groups with the impulsive rate of growth.}
	\label{Fig 16}
\end{figure}

\begin{figure}[ht]
	\centering
	\includegraphics[scale=0.86]{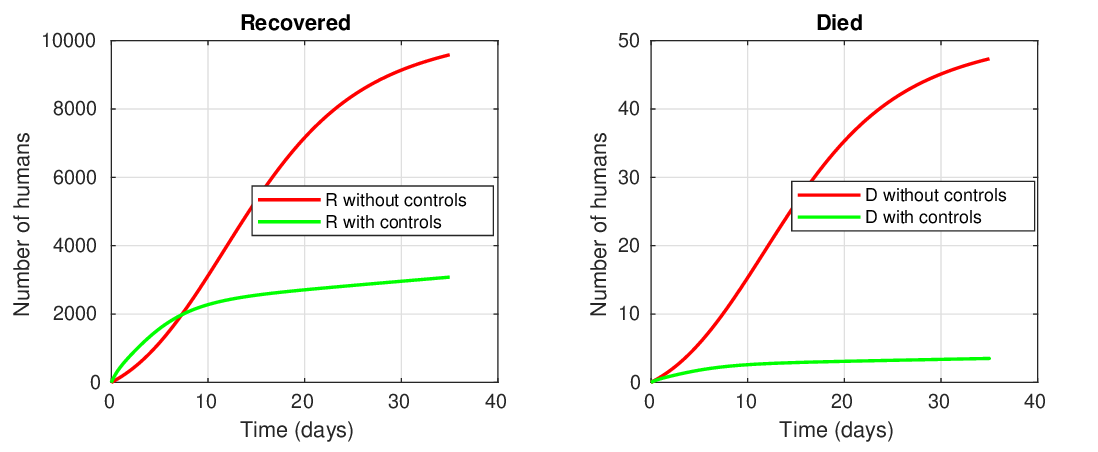} 
	\caption{Changes in the Recovered and Death groups with the impulsive rate of growth.}
	\label{Fig 18}
\end{figure}

\begin{figure}[ht]
	\centering
	\includegraphics[scale=0.86]{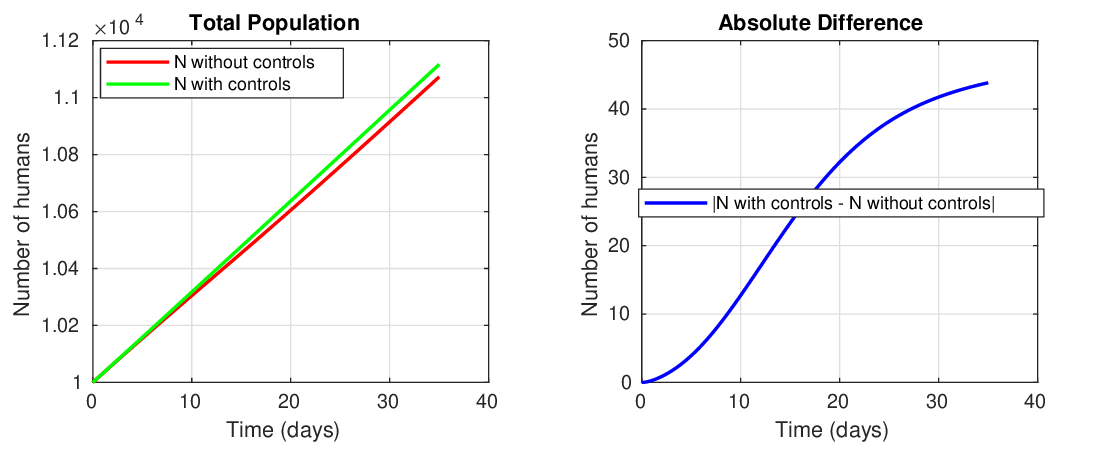} 
	\caption{Changes in the total population with the impulsive rate of growth.}
	\label{Fig 20}
\end{figure}

\begin{figure}[ht]
	\centering
	\includegraphics[scale=0.86]{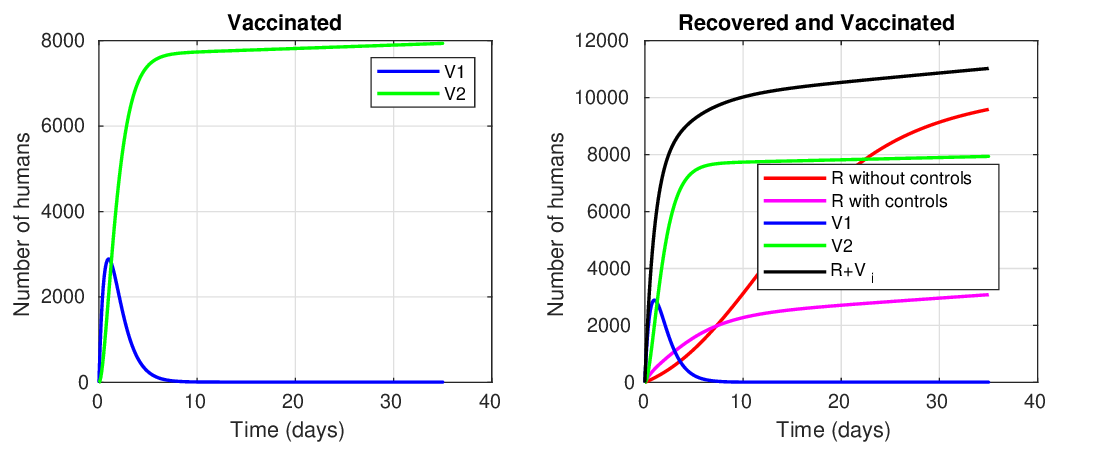} 
	\caption{Comparison between recovered and vaccinated population and a total of them with and without controls.}
	\label{Fig 22}
\end{figure}

\section{Comparison with other diseases} \label{section 8}
\noindent

This section is divided into three subsections. In the first subsection, we discuss the Ebola disease with and without controls. We follow the same procedure for the Influenza (H2N2) disease. Lastly, we compare the three diseases (COVID-19, Ebola, and Influenza). The parameters considered in this comparison are presented in the following table. For COVID-19, the parameters are taken from Li et al. \cite{Li} and Riou et al. \cite{Riou}. For Ebola disease, the parameters are taken from Althaus et al. \cite{Althaus}. For Influenza disease, the parameters are taken from Arino et al. \cite{Arino}.
\begin{table}[h!]
	\centering
	\label{table2_Num}
	\begin{tabular}{cccccc} 
	\hline
	Parameter & Value (COVID-19) & Value (Ebola) & Value (Influenza)\\
	\hline
	$z$ & 0.1 & 0.76 & 0.667 \\ 
	$\eta$ & 0.3 & 0.178 & 0.244 \\ 
	$k$ & 0.54 & 0.0023 & 0.526 \\
	$\alpha$ & 0.995 & 0.26 & 0.98 \\
	$p$ & 0.02 & 0.02 & 0.9 \\
	$q$ & 0.5 & 0.5 & 0.5 \\ 
	$f$  & 0.3 & 0.178 & 0.244 \\
	$\gamma_1$ & 1 & 1 & 1 \\
	$\mu$ & 1 & 1 & 1 \\ 
	$\varepsilon$ & 0 & 0 & 0 \\ 
	\hline
	\end{tabular}
\end{table}

\subsection{Ebola disease.}
\noindent

For Ebola disease, as shown in figures \ref{Fig 26}  and \ref{Fig 26-d}, the application of controls reduces the number of susceptible individuals to zero within only five days, which is not the case when controls are absent, as their number will never reach zero or may take more time to do so. Regarding the exposed population, we can observe a significant increase in their number when controllers are absent, whereas when controls are applied, there is a substantial decrease in their number. This reduction is appropriate because the rate $k$ is small, indicating that the exposed population slowly moves to the groups of infected and asymptomatic individuals. Additionally, since the population is dynamic, new individuals are continuously being added (through migration or travel), preventing certain groups from declining rapidly. However, they eventually reach zero at some point, as shown in remark \ref{rem 4}. Concerning the infected population, it is evident that when controls are applied, their number decreases and eventually reaches zero. Conversely, when controls are absent, their number never reaches zero due to the sudden population increase. For the recovered population, a substantial difference is observed between the number of individuals who recover from the virus when controls are applied and when they are absent. With controls, we can recover about 1300 individuals from the virus, whereas without controls, we recover only about 400 individuals. In the absence of controls, approximately 1050 persons die due to the infection. Implementing controls allows us to save more than 860 persons' lives. As a consequence, the Ebola virus is severe and has a high case fatality rate. However, by immunizing those susceptible to the infection and treating those already infected, we can effectively eliminate the disease. 
\subsection{Influenza disease.}
\noindent

For Influenza disease, figures \ref{Fig 27} and \ref{Fig 27-d} demonstrate the effectiveness of the controller in managing the susceptible population by transferring them to the group of vaccinated individuals, with only a few moving to the group of exposed individuals. Furthermore, in the presence of controls, the number of exposed and infected individuals continuously decreases until reaching zero from the first day. However, the proportion of asymptomatic individuals with the controller has also decreased to its lowest point, yet it will take more time to reach zero compared to other groups due to the absence of a direct controller. Moreover, when controls are present, the number of recovered individuals is lower compared to when no controller is used. This is because in the absence of a controller, there are more infected individuals. The number of vaccinated individuals reaches about 8000 persons, while the number of deaths remains stable at 7 persons, accounting for the additional population. As a consequence, immunizing more of the population against Influenza disease enables us to save more lives.

\subsection{Comparison between COVID-19, Ebola, and Influenza.}
\noindent

Figures \ref{Fig 28} and \ref{Fig 28-d} show that in the presence of controls, population susceptible to infection goes to zero in around five days. For exposed population in Ebola disease, their number increase before starting to decrease to zero because of the rate $k$. Unlike, for COVID-19 and Influenza, the number of exposed population decrease to zero quickly (in around $15$ days). For asymptomatic population, zeroing their number may take more time since there is no direct controls in this groups. It is the same for infected population even there is a direct control in this group, but this is not enough the fact that others population relocate this group away from asymptomatic and exposed groups. For recovered population, their number is smaller for Ebola disease than that for COVID-19 and influenza. This is acceptable because of the fatality rate of Ebola. Then, a few of them will recovered from of the virus when the rest of them will die because of the infection, even in the presence of controls.
\section{Conclusion } \label{section 7}
\noindent

In this paper, we have considered a SEIAR dynamic epidemic model adding new groups called vaccinated population ($V_1$, $\ldots$, $V_n$), denoted as (VS-EIAR). The aim was to control the propagation of the infection by decreasing the number of susceptible, exposed, asymptomatic, and infected populations using optimal control theory. Our strategy was based on two control actions: the first one involved vaccinating the population susceptible to the infection, and the second one involved providing treatment to those who are infected. The Pontryagin's Maximum Principle (PMP) was used to establish the optimal controls and the finite optimal time. To enhance the study's relevance, we considered an impulsive dynamic pandemic model that accounts for immigration and travel. The impulsive VS-EIAR epidemic model can also be controlled using our control strategy. The theoretical analysis was complemented by numerical simulations to test the viability of our method for controlling the spread of COVID-19. Additionally, a comparison is provided to compare between three different diseases (COVID-19, Ebola, and Influenza) under both impulsive and non-impulsive scenarios.

One area for future work is to extend our model to account for populations vaccinated with different types of vaccines. Incorporating various vaccination strategies and their respective efficacy rates could provide a more comprehensive understanding of disease control. Furthermore, the accuracy of the VS-EIAR model could be impacted by its failure to account for vaccinated individuals who can still be susceptible, infected, exposed, or asymptomatic, as vaccination does not provide absolute immunity. Addressing this limitation in future iterations of the model would improve its robustness and predictive power.

\begin{figure}[ht]
	\centering
	\includegraphics[scale=0.86]{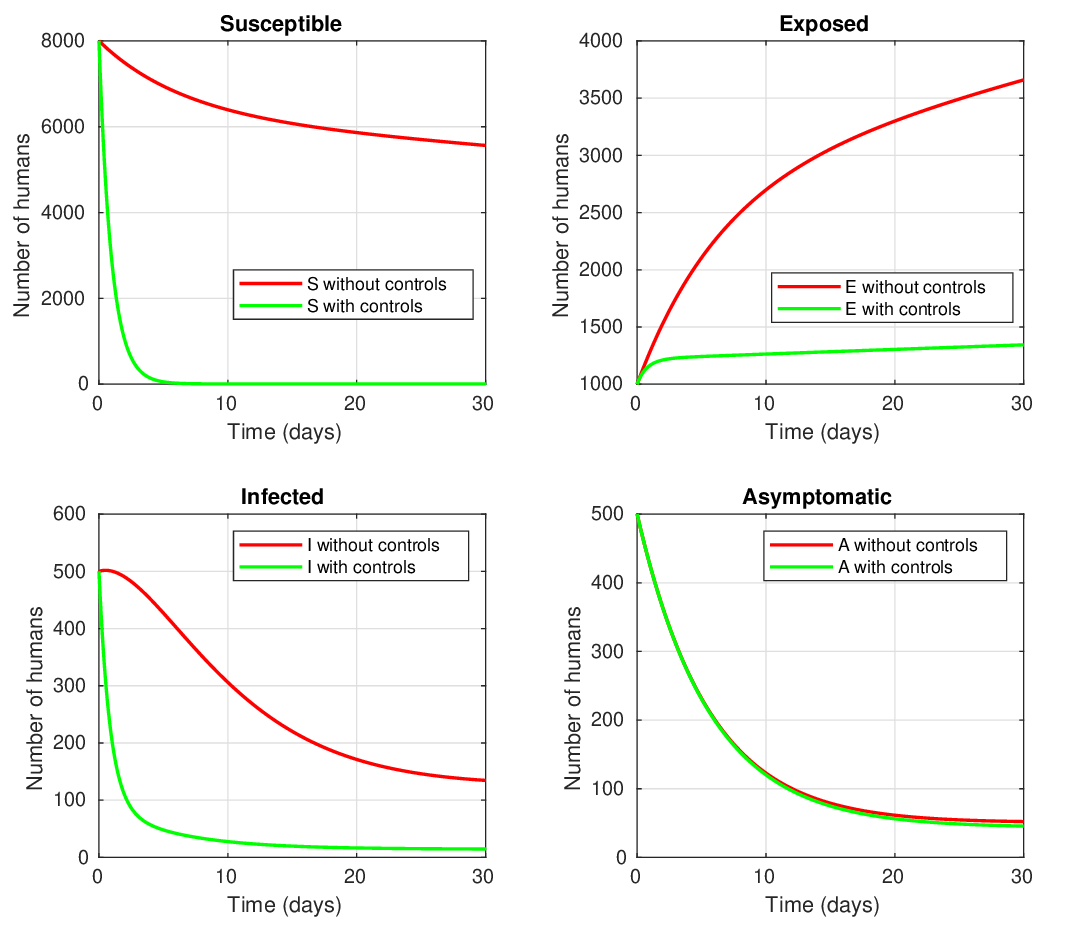}
	\caption{Changes of ($S(t)$, $E(t)$, $A(t)$, $I(t)$) with the impulsive rate of growth in Ebola.}\label{Fig 26}
\end{figure}

\begin{figure}[ht]
	\centering
	\includegraphics[scale=0.86]{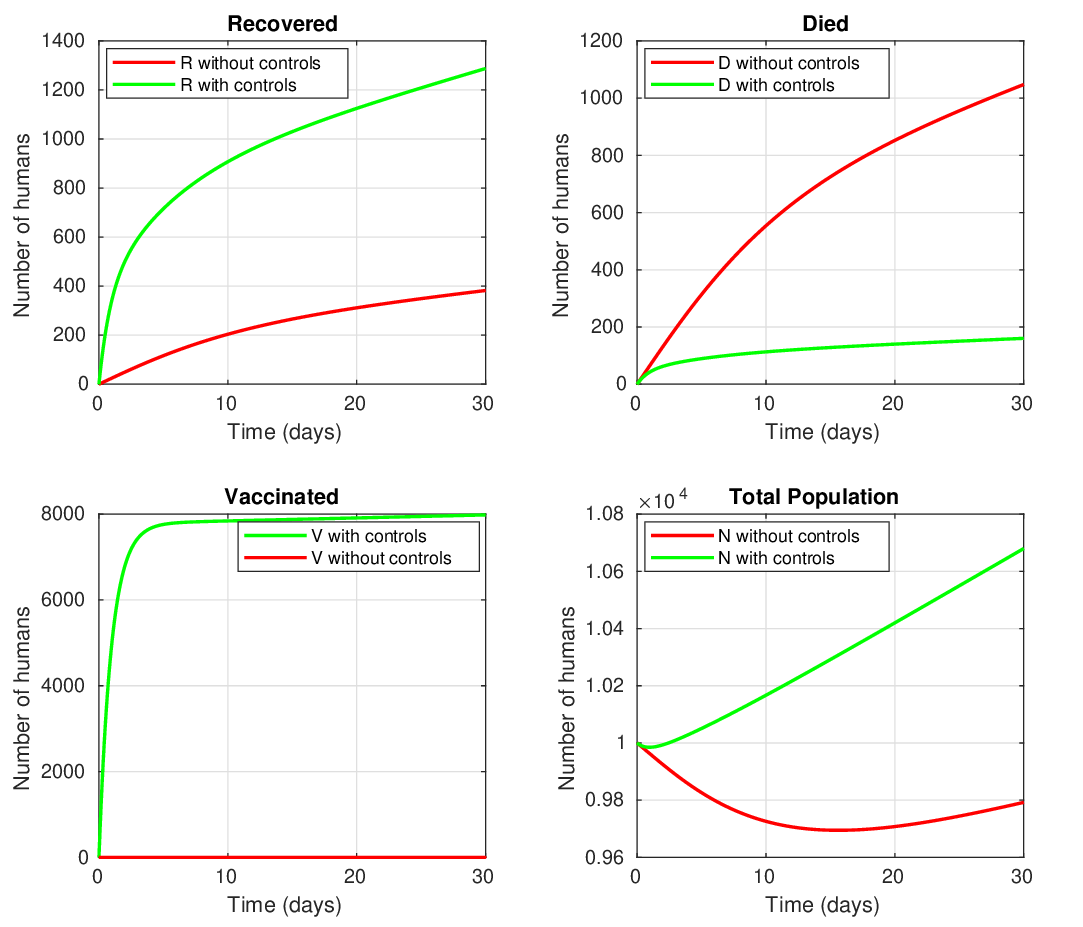}
	\caption{Changes of ($R(t)$, $V(t)$ $D(t)$, $N(t)$) with the impulsive rate of growth in Ebola.}\label{Fig 26-d}
\end{figure}

\begin{figure}[ht]
	\centering
	\includegraphics[scale=0.86]{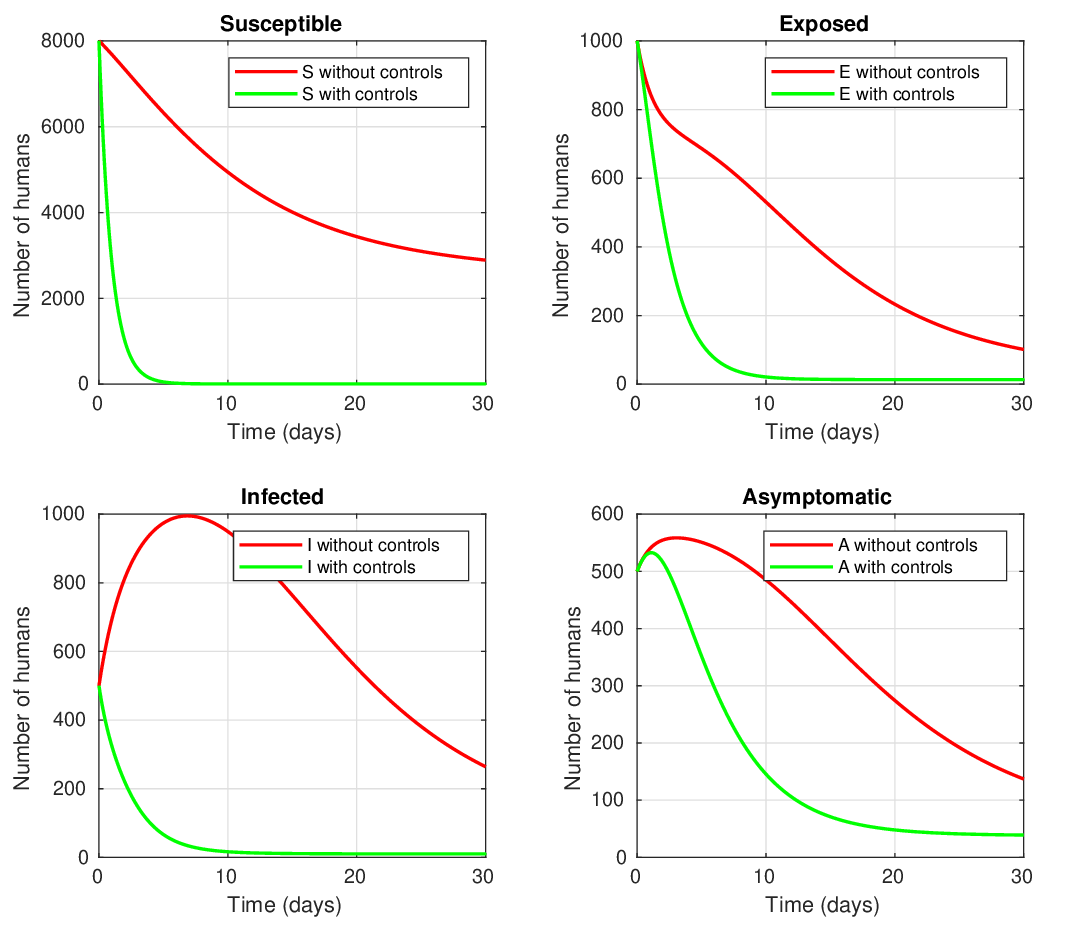}
	\caption{Changes of ($S(t)$, $E(t)$, $A(t)$, $I(t)$) with the impulsive rate of growth in Influenza.}\label{Fig 27}
\end{figure}

\begin{figure}[ht]
	\centering
	\includegraphics[scale=0.86]{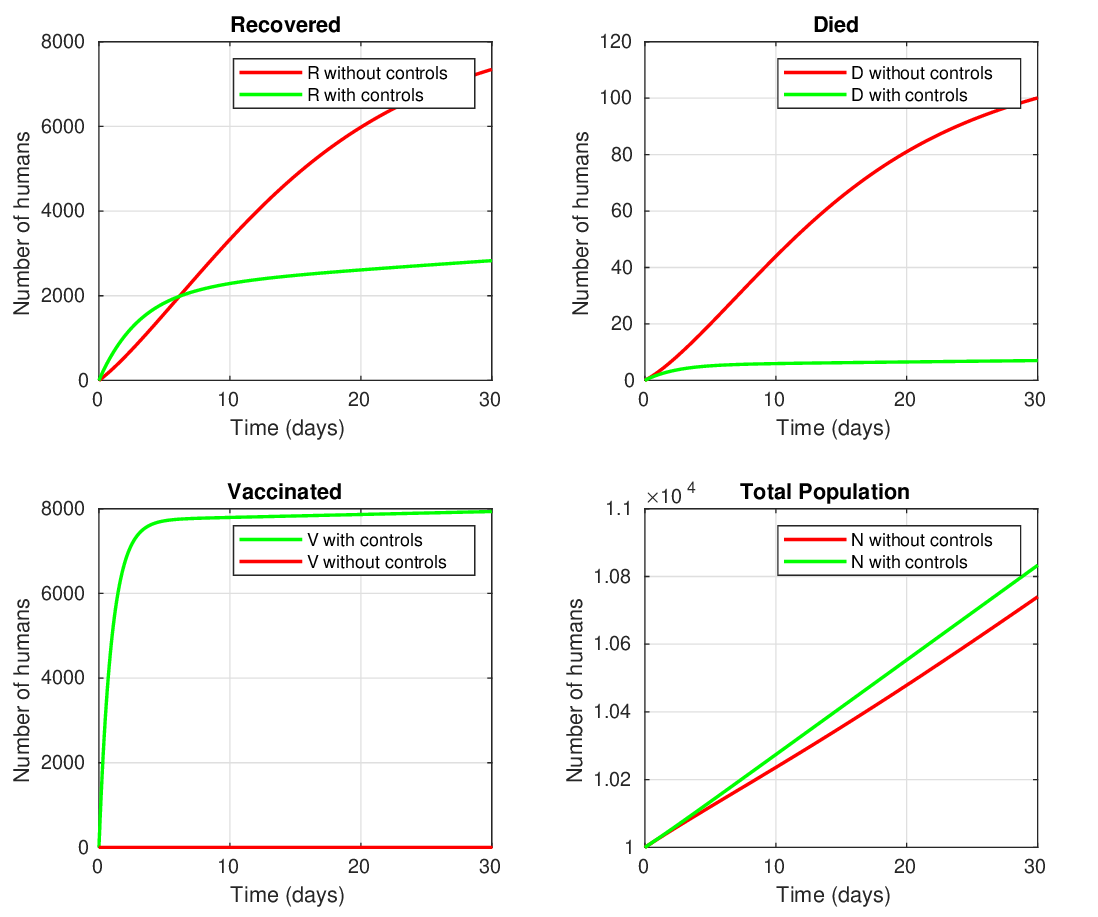}
	\caption{Changes of ($R(t)$, $V(t)$ $D(t)$, $N(t)$) with the impulsive rate of growth in Influenza.}\label{Fig 27-d}
\end{figure}

\begin{figure}[ht]
	\centering
	\includegraphics[scale=0.86]{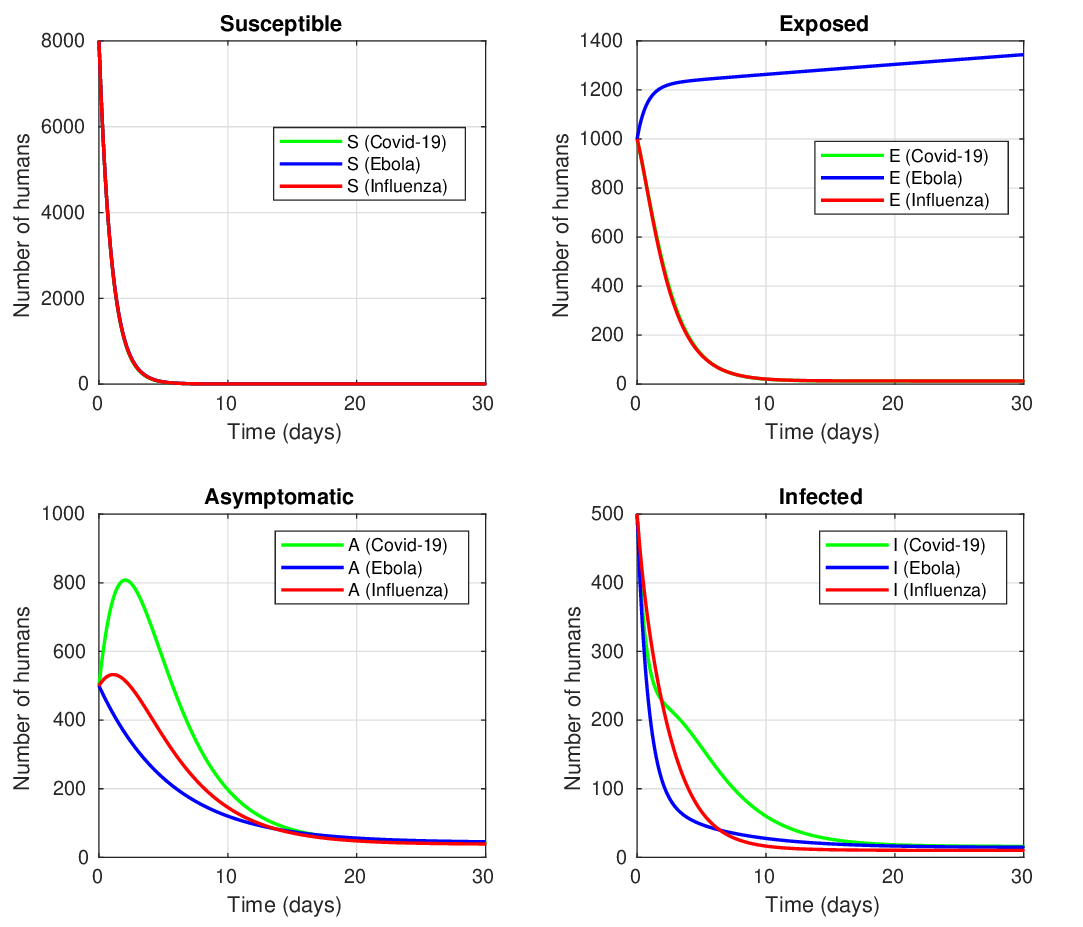}
	\caption{Comparison in ($S(t)$, $E(t)$, $A(t)$, $I(t)$) with impulsive for COVID-19, Ebola, and Influenza.}\label{Fig 28}
\end{figure}

\begin{figure}[ht]
	\centering
	\includegraphics[scale=0.86]{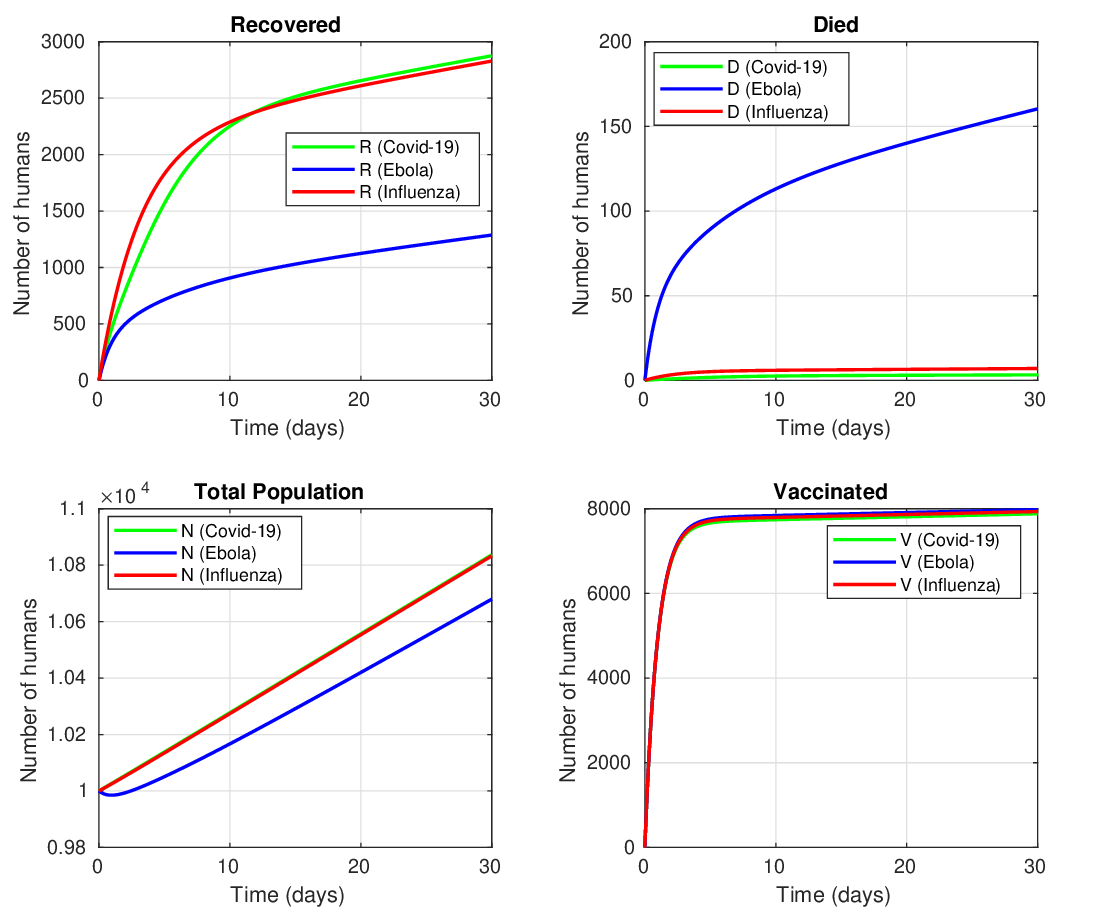}
	\caption{Comparison in ($R(t)$, $V(t)$ $D(t)$, $N(t)$) with impulsive for COVID-19, Ebola, and Influenza.}\label{Fig 28-d}
\end{figure}

\section{Appendix} \label{the Appendix}
\begin{proof}[\textbf{Proof of Theorem \ref{theorem 2.1}.}] Let $x(t)=\left(S(t),E(t),A(t),I(t),R(t),D(t), V_1(t),\ldots,V_n(t) \right)^{T}$. Equation (\ref{eq 1}) can be written as follows:
	\begin{equation}
	\left\{\begin{array}{l}
	\dot{x}(t)=F(t,x(t)), \quad t\geq 0\\
	x(0)=x_0,
	\end{array}\right.
	\label{eq 9.1}
	\end{equation}
	where $F(t,x(t))\equiv F(x(t),w(t))$ for $w(t)=(v(t),u(t))^{T}$ in which $F(t,\cdot)=\left(F_1(t,\cdot),\ldots,F_{n+6}(t,\cdot)\right)$ is defined on $\mathbb{R}^{n+6}$ by 
	
	\begin{equation*}
	\left.\begin{array}{l}
	F_1(t,y)=-(\beta[\varepsilon y_2+(1-q)y_4+\mu y_3]+\gamma_1v(t))y_1,\\[10pt]
	F_2(t,y) =\beta[\varepsilon y_2+(1-q)y_4+\mu y_3]y_1-ky_2+\sum_{i=1}^{n-1}\delta_iy_i,\\[10pt]
	F_3(t,y)=(1-z)ky_2-\eta y_3,\\[10pt]
	F_4(t,y)=zky_2+(1-p)\eta y_3-fy_4-u(t)y_4,\\[10pt]
	F_5(t,y) =\alpha fy_4+u(t)y_4+p\eta y_3,\\[10pt]
	F_6(t,y)=(1-\alpha)f y_4,\\[12pt]
	F_7(t,y)=\gamma_1 v(t)y_1-\gamma_{2}v(t)y_7-\delta_{1} y_7\\[10pt]
	F_{i+6}(t,y)=\gamma_iv(t)y_{i+5}-\gamma_{i+1}v(t)y_{i+6}-\delta_iy_{i+6}, \quad \text{for } i=2,3,\ldots,n-1,\\[10pt]
	F_{n+6}(t,y)=\gamma_nv(t)y_{n+5},
	\end{array}\right.
	\end{equation*}
	for $t\geq 0$ and $y=(y_1,\ldots,y_{n+6})\in\mathbb{R}^{n+6}_{+}$. To show that equation \eqref{eq 9.1} have a unique solution defined on an interval $[0,t_{max})$ for $t_{max}\leq +\infty$, it is sufficient to show that equation:
	\begin{equation}
	x(t)=x_0+\displaystyle\int_{0}^{t}F(s,x(s))ds \hspace{0.1cm} \text{ for } \hspace{0.1cm} t\geq 0,
	\label{eq int 9.2}
	\end{equation}
	has a unique solution $x(\cdot)$ defined on $[0,t_{max})$. In fact, the function $t\to F(t,Y)$ is almost everywhere continuous on $\mathbb{R}^{+}$ for each $Y\in \mathbb{R}^{n+6}_{+}$. Since $Y\to F(t,Y)$ is $\mathcal{C}^{1}(\mathbb{R}^{n+6})$, it follows that $F(t,\cdot)$ is locally Lipschitz with respect to the second  argument. As a consequence, there exists a unique function $x(\cdot)$ defined on an interval $[0,t_{max})$ for $t_{max}\leq +\infty$ that satisfy equation \eqref{eq int 9.2}.

	Next, we demonstrate the positivity of the solutions. By Definition 4.1 from \cite{Haddad}, we can show that $F$ is essentially nonnegative. Following Proposition 4.1 from \cite{Haddad}, we deduce that the solutions of equation \eqref{eq 1} are positive.
	 
	Finally, we show the boundedness of solutions. Since $\dot{N}(t)=(\alpha-1)I(t)\leq 0$, it follows that $N(\cdot)$ is decreasing, which implies that $0\leq N(t)\leq N_0$ for $t\in [0,t_{max})$. Then, solutions of equation \eqref{eq 1} are bounded. As a consequence $t_{max}=+\infty$.
\end{proof}

\begin{proof}[ \textbf{Proof of Theorem \ref{thm 3.2}.}]
	The optimal controls can be calculated by 
	\begin{center}
		$\nabla_{u(t)}^{H}=\dfrac{\partial H}{\partial u(t)}=0$ \quad and \quad $\nabla_{v(t)}^{H}=\dfrac{\partial H}{\partial v(t)}=0$.
	\end{center}
	Therefore,
	\begin{center}
		$u^{*}(t)=\dfrac{I^{*}(t)\left[ p_4(t)-p_5(t)\right]}{\sigma_0}$,
	\end{center}
	and
	\begin{center}
		$v^{*}(t)=\dfrac{\left[ \gamma_1S^{*}(t)[p_1(t)-q_1(t)]+\gamma_{2}q_1(t)V_1^*(t)+\sum\limits_{i=2}^{n-1}q_i(t)[\gamma_{i+1}V_i^*(t)-\gamma_{i}V^*_{i-1}(t)]\right]}{\left(\sum\limits_{i=1}^{n}\sigma_i\gamma_i^2\right)}$,
	\end{center}
	for $t\in[0,\tau^*]$. Since $u^{*}\in U_{ad}^{1}$ and $v^{*}\in U_{ad}^{2}$, it follows that
	\begin{equation*}
	u^{*}(t)=\max\left\{\min\left\{\dfrac{I^{*}(t)\left[ p_4(t)-p_5(t)\right]}{\sigma_0},1 \right\},0\right\},
	\end{equation*}
	and
	\begin{equation*}
	v^{*}(t)=\max\left\{\hspace{-0.1cm}\min\left\{\dfrac{W(t)}{\left(\sum\limits_{i=1}^{n}\sigma_i\gamma_i^2\right)},\dfrac{1}{\gamma_1} \hspace{-0.1cm}\right\},0\hspace{-0.1cm}\right\},
	\end{equation*}
	where  
	\begin{center}
		$W(t)=\dfrac{\left[ \gamma_1S^{*}(t)[p_1(t)-q_1(t)]+\gamma_{2}q_1(t)V_1^*(t)+\sum\limits_{i=2}^{n-1}q_i(t)[\gamma_{i+1}V_i^*(t)-\gamma_{i}V^*_{i-1}(t)]\right]}{\left(\sum\limits_{i=1}^{n}\sigma_i\gamma_i^2\right)}$,
	\end{center}
	for $t\in[0,\tau^*]$. The optimal final time $\tau^{*}$ can be calculated by
	\begin{equation*}
	H\left(X^{*}(\tau^*),v^{*}(\tau^*),u^{*}(\tau^*),P^{*}(\tau^*),Q^{*}(\tau^*),\tau^{*}\right)+\dfrac{\partial \mathcal{M}(\tau^{*})}{\partial t}=0.
	\end{equation*}
	The rest of the proof follows in \eqref{eqn 4}-\eqref{eqn 4-d}.
\end{proof}
\begin{proof}[\textbf{Proof of Theorem \ref{thm 4}.}]  Let $x(t)=\left(S(t),E(t),A(t),I(t),R(t),D(t), V_1(t),\ldots,V_n(t) \right)^{T}$, and $\rho_i(\cdot)$ $(\text{for } i=1,\ldots,4)$ be the functions defined by
	\begin{equation*}
	\rho_i(t)=	\left\{\begin{array}{l}
	\lambda_i(t_k) \text{ if } t=t^{+}_k\\
	0 \text{ if } t\neq t^{+}_k.
	\end{array}\right.
	\end{equation*}
	Equation \eqref{equ 3} can be written as 
	\begin{equation}
	\left\{\begin{array}{l}
	\dot{x}(t)=L(t,x(t)), \quad t\geq 0\\
	x(0)=x_0,
	\end{array}\right.
	\label{eq 9.2}
	\end{equation}
	where $L(t,\cdot)=\left(L_1(t,\cdot),\ldots,L_{n+6}(t,\cdot)\right)$ is defined on $\mathbb{R}^{n+6}$ by
	\begin{equation*}
	L_{i}(t,y)=\left\{\begin{array}{l}
	F_{i}(t,y)+\rho_i(t)y_i \hspace{0.1cm} \text{ for } \hspace{0.1cm} i=1,\ldots,4,\\
	F_{i}(t,y) \hspace{0.1cm} \text{ for } \hspace{0.1cm} i=5,\ldots,n+6,
	\end{array}\right.
	\end{equation*}
	for $t\geq 0$ and $y=(y_1,\ldots,y_{n+6})\in\mathbb{R}^{n+6}_{+}$. The function $t\to L(t,Y)$ is locally $\mathbb{L}^{1}-$integrable on $\mathbb{R}^{+}$ for each $Y\in \mathbb{R}^{n+6}_{+}$. Since $L(t,\cdot)$ is $\mathcal{C}^{1}(\mathbb{R}^{n+6})$, it follows that $L(t,\cdot)$ is Locally Lipschitz.  As a consequence, equation \eqref{eq 9.2} have a unique solution $x(\cdot)$ defined on an interval $[0,t_{max})$ for $t_{max}\leq +\infty$, and satisfying the following integral equation:
	\begin{equation*}
	x(t)=x_0+\displaystyle\int_{0}^{t}L(s,x(s))ds \hspace{0.1cm} \text{ for } \hspace{0.1cm} t\in [0,t_{max}).
	\end{equation*}
	
	To show the positivity of solutions, we follow the same approach as presented in the proof of Theorem \ref{theorem 2.1}. Finally, we show the boundedness of solutions and  $t_{max}=+\infty$. For this reason, we discuses the following cases:
	\begin{enumerate}
		\item[] \textbf{case 1:} if $t_{max}<t_1$, we have $\dot{N}(t)=(\alpha-1)I(t)\leq 0$ for $t\in[0,t_{max})$, then $N(t)\leq N_0$ for $t\in[0,t_{max})$.
		\item[] \textbf{case 2:} if $t_{max}\in[t_1,t_2[$, we have $\dot{N}(t)=(\alpha-1)I(t)\leq 0$ for $t\in[0,t_1[$, then $N(t)\leq N_0$ for $t\in[0,t_1[$. If $t=t_1$, we have $N(t_1)=N(t_1^{-})+N(t_1^{+})-N(t_1^{-})\leq N_0+\left[N(t_1^{+})-N(t_1^{-})\right]$. Then, $N(t)\leq N_0+\left[N(t_1^{+})-N(t_1^{-})\right]$ for $t\in[0,t_1]$. If  $t\in]t_1,t_{max})$, we have  $\dot{N}(t)=(\alpha-1)I(t)\leq 0$, then $N(t)\leq N(t_1)\leq N_0+\left[N(t_1^{+})-N(t_1^{-})\right]$ for $t\in[t_1,t_{max})$. As a consequence, in that case, $N(t)\leq N_0+\left[N(t_1^{+})-N(t_1^{-})\right]$ for $t\in[0,t_{max})$.
		\item[] \textbf{case 3:} without loss of generality, we can assume that $t_{max}>t_p$. For $t\in]t_1,t_2[$, we have $\dot{N}(t)=(\alpha-1)I(t)\leq 0$, then $N(t)\leq N(t_1)\leq N_0+\left[N(t_1^{+})-N(t_1^{-})\right]$. If $t=t_2$, we have $N(t_2)=N(t_2^{-})+N(t_2^{+})-N(t_2^{-})\leq N(t_1)+\left[N(t_2^{+})-N(t_2^{-})\right]$. Hence, $N(t)\leq N_0+\sum\limits_{i=1}^{2}\left[N(t_i^{+})-N(t_i^{-})\right]$ for $t\in[0,t_2]$. Similarly, we obtain $N(t)\leq N_0+\sum\limits_{i=1}^{p}\left[N(t_i^{+})-N(t_i^{-})\right]$ for $t\in[0,t_{max})$.
	\end{enumerate}
	We conclude that solutions of equation \eqref{equ 3} are bounded and must be defined on $\mathbb{R}^{+}$, implying that $t_{max}=+\infty$.
\end{proof}
\section*{Acknowledgments}
 We thank the anonymous reviewers for their valuable time, insightful comments, and helpful suggestions.
 \vspace{-0.25cm}
\section*{Financial disclosure}
The authors declare that they have no known competing financial interests or personal relationships that could have appeared to influence the work reported in this paper.
\vspace{-0.15cm}
\section*{Conflict of interest}
This work does not have any conflicts of interest.

\end{document}